\documentclass[english,preprint]{elsarticle}

\usepackage{graphics}
\usepackage{graphicx}
\usepackage{amsmath,amsthm,amssymb,babel}
\usepackage{color}
\usepackage{multirow}
\usepackage{afterpage}
\usepackage{placeins}
\usepackage[justification=centering]{caption}
\usepackage{epsfig}
\usepackage{url}
\usepackage[nobottomtitles]{titlesec}
\usepackage{lineno}
\usepackage{algorithm}
\usepackage[noend]{algpseudocode}
\usepackage[table]{xcolor}
\usepackage{mathtools}
\usepackage{hyperref}

\usepackage{color}

\usepackage{subfig}

\newlength\longarc

\newcommand{\R}{{\if mm {\rm I}\mkern -3mu{\rm R}\else \leavevmode
		\hbox{I}\kern -.17em\hbox{R} \fi}}
\newcommand{\blambda}{\boldsymbol\lambda}

\newcommand{\bB}{\mbox{\boldmath{$B$}}}
\newcommand{\bdelta}{\mbox{\boldmath{$\delta$}}}
\newcommand{\bPi}{\mbox{\boldmath{$\Pi$}}}

\newcommand{\bC}{\mbox{\boldmath{$C$}}}

\newcommand{\bu}{\mbox{\boldmath{$u$}}}

\newcommand{\bv}{\mbox{\boldmath{$v$}}}

\newcommand{\bx}{\mbox{\boldmath{$x$}}}

\newcommand{\by}{\mbox{\boldmath{$y$}}}

\newcommand{\bF}{\mbox{\boldmath{$F$}}}
\newcommand{\bI}{\mbox{\boldmath{$I$}}}
\newcommand{\bg}{\mbox{\boldmath{$g$}}}
\newcommand{\fb}{\mbox{\boldmath{$f$}}}

\newcommand{\bn}{\mbox{\boldmath{$n$}}}
\newcommand{\bsigma}{\mbox{\boldmath{$\sigma$}}}
\newcommand{\bSigma}{\mbox{\boldmath{$\Sigma$}}}

\newcommand{\btau}{\mbox{\boldmath{$\tau$}}}
\newcommand{\bvarphi}{\mbox{\boldmath{$\varphi$}}}
\newcommand{\bvarepsilon}{\mbox{\boldmath{$\varepsilon$}}}
\newcommand{\bnu}{\mbox{\boldmath{$\nu$}}}

\newcommand{\betau}{\mbox{\boldmath{$\tau$}}}
\newcommand{\bzero}{\mbox{\boldmath{$0$}}}
\newcommand{\bgamma}{\mbox{\boldmath{$\gamma$}}}

\newtheorem{theorem}{Theorem}[section]

\newtheorem{definition}[theorem]{Definition}
\newtheorem{proposition}[theorem]{Proposition}

\numberwithin{equation}{section}

\begin{document}
\vskip 5mm
\author{\it  Mika\"el Barboteu$^*$ 
	, Francesco Bonaldi$^*$,
 David Danan$^\#$ and Soad El-Hadri$^\&$\\[5mm]
 {\small barboteu@univ-perp.fr},\ {\small francesco.bonaldi@univ-perp.fr},\
  \small{david.danan@irt-systemx.fr}, \	\small{elhadrisoadd@gmail.com},\\[0.3cm]
	$^*$\small{Laboratoire de Mod\'elisation Pluridisciplinaire et Simulations}\\[-0mm]{\small Universit\'e de
		Perpignan Via Domitia}\\[-0mm] \small{52 Avenue Paul Alduy, 66860 Perpignan,
		France}\\[0.3cm]
	 $^\#$\small{Institut de Recherche Technologique SystemX}\\[-0mm]{\small Centre d'int\'egration Nano-INNOV}\\[-0mm] \small{8, Avenue de la Vauve, 91127 Palaiseau CEDEX
		France}\\[0.3cm]
	 $^\&$\small{École Nationale de l'Aviation Civile}\\[-0mm] \small{7 Avenue Edouard Belin, 31400 Toulouse}\\[0.3cm]		
}

\title{An Improved Normal Compliance Method for Dynamic Hyperelastic Problems with Energy Conservation Property}

\date{}

\begin{abstract}
	 The purpose of this work is to present an improved energy conservation method for hyperelastodynamic contact problems based on specific normal compliance conditions. In order to determine this Improved Normal Compliance (INC) law, we use a Moreau--Yosida $\alpha$-regularization to approximate the unilateral contact law. Then, based on the work of Hauret--LeTallec \cite{hauret2006energy}, we propose in the discrete framework a specific approach allowing to respect the energy conservation of the system in adequacy with the continuous case. This strategy (INC) is characterized by a conserving
behavior for frictionless impacts and admissible dissipation for
friction phenomena while limiting penetration. Then, we detail the numerical treatment within the framework of the semi-smooth Newton method and primal-dual active set strategy for the normal compliance conditions with friction. We finally provide some numerical experiments to bring into light the energy conservation and the efficiency of the INC method by comparing with different classical methods from the literature throught representative contact problems.
	 
\end{abstract}

\maketitle

\bigskip \noindent {\bf AMS Subject Classification\,:} 74M15, 74M20, 74M10, 74B20, 74H15, 74S30, 49M15, 90C53, 70F40, 70-08, 70E55, 35Q70




\section{Introduction} \label{int}
Dynamic problems involving contact  between deformable bodies or between a body and a foundation still remain today an important subject of study for mathematical and numerical analysis. In the literature, many references offer different approaches for the usual contact conditions with friction \cite{Ac,CDR98,CFHMPR,kpr,lebon2003,pietrzak1999large,laur2002,RCL88,SM,wr}. As a matter of fact, the equations resulting from frictional contact problems are difficult to solve both mathematically and numerically due to their inherent non-linearity and the non-smoothness. In order to handle these issues, several methods have already been successfully tested. They rely on handling the original conditions by normal compliance methods \cite{KO,OK,MartOden} or by other relevant methods such as the quasi-augmented lagrangian method \cite{Ac}, the bi-potential method \cite{dSF91, dumont13}, the conjugate gradient method \cite{laur2002,wr}, Uzawa method \cite{RCL88,uzawa} and Nitsche finite element method  \cite{Chouly14,Chouly,CFHMPR} and references therein. Moreover, Newton's semi-smooth method with Primal-Dual Active Set (PDAS) procedure appears as one of the most relevant methods to solve contact problems with friction because of their efficiency and their simplicity of implementation \cite{hint2,wohlcoulomb,wohl}. Some works have been dedicated to studying the efficiency of PDAS methods, as well as to solve linear elasticity problems with unilateral boundary conditions \cite{hint2,hint,ABD1}, or even to solve  non-linear multi-body contact problems in elastodynamics \cite{wohlcoulomb,wohl,ABCD2}.
 
However, when considering impact problems whether in small or large deformations, even the standard implicit schemes ($\theta$-method, Newmark schemes, midpoints or Hilber-Hughes-Taylor type methods \cite{hht}) lose their unconditional stability, which leads either to an energy explosion or to a numerical dissipation which is  neither realistic nor acceptable mechanically, as explained in \cite{hht,simo,gonz,arro}. Therefore, it is necessary to use appropriate implicit schemes providing energy conservation type properties.
To this purpose, many references \cite{hauret2006energy,laur2002,simo,gonz,arro, Lach,AP,LL,ayyad2009formulation,Acary13,Acary15,barboteu2015hyperelastic} propose relevant approaches to solve these impact problems with balance energy properties.

This work proposes two traits of novelty. The first one arises from the improved frictional contact model we consider, which provides intrinsic energy-controlling properties. Contact is modeled with a general Moreau--Yosida  regularization \cite{conj,MoYosi,regular} of the unilateral condition. The Moreau--Yosida regularization  seems to be an appropriate tool to find a regular contact model (Normal Compliance) which respects the kinetic energy of the system while preserving the non-interpenetration of contact.  The discrete Improved Normal Compliance (INC) will be well suited to respect energy conservation in adequation with the continuous framework.  The second trait of novelty consists in the analysis and the implementation of an energy-controlling method, based on a semi-smooth Newton method combined with an active set method via the complementarity functions for normal compliance with friction models. In summary, the major novelty comes from the combination of the Active Set approach with the Improved Normal conformance method and the fact that such a combination is also useful and relevant for the conservation of the energy in hyper-elastodynamics. Based on representative examples from the literature \cite{laur2002,ayyad2009formulation,khenous2008mass}, we study and analyze numerically this Improved Normal Compliance method for dynamic hyperelastic problems with the main objective of respecting the conservation of energy and the non-interpenetration condition during impacts.

The remainder of the article is organized as follows. In Section \ref{s2}, we present and explain the physical framework and the mathematical model studied, then we recall the formulation of hyperelastic problems with frictional contact. We present in detail various contact models as unilateral contact, persistent contact, and normal compliance conditions. After formulating the strong and variational problems, we detail the energy conservation and dissipation properties in the continuous case using the specific properties of the normal compliance conditions. Section \ref{s3} is devoted to the discretization of the hyperelastodynamic problem with contact and the approach (INC) adapted to respect the energy conservation in adequation to the continuous framework. In Sections \ref{s4} and \ref{s5}, we propose an innovative, fast and efficient Primal Dual Active Set (PDAS) method to solve a hyperelastodynamic problem with Improved Normal Compliance and Coulomb friction. Contact and friction conditions are realized by applying an Active Set strategy via a non-linear complementarity function based on the semi-smooth Newton iterative scheme.
In the last section, we provide the numerical experiments of the hyperelastodynamic problems in contact with and without friction carried out with a dynamic elastic ball, then a hyperelastic ring both launched in the direction of a rigid foundation. We present  comparative studies between different numerical energy conservation methods during the impacts of the systems [1, 2, 31, 32].
%
%
\section{Hyperlastic problems for low velocity impact with friction}\label{s2}
\subsection{Hyperlastic framework}
A hyperelastic body occupies a bounded domain $\Omega\subset \mathbb{R}^d$ $(d=1,2,3)$ with a continuous Lipschitz boundary $\Gamma$, divided into three disjoint measurable parts $\Gamma_1 $, $\Gamma_2$, and $\Gamma_3$. We denote by $\bx=(x_i)$ the point in $\Omega\cup\Gamma$ used and we designate by ${\bn}=(n_i)$ the unit outward normal over $\Gamma$. The indices $i$, $j$, $k$, $l$ vary between $1$ and $d$ ($d$ is the space dimension), and unless otherwise specified, the summation convention on repeated indices is employed. We denote by $\mathbb{M}^d$ the space of second-order tensors on $\mathbb{R}^d$ or, equivalently, the space of square matrices of order $d$. The scalar product and the norm on $\mathbb{R}^d$ and $\mathbb{M}^d$ are defined by
\begin{eqnarray*}
	&\bu\cdot \bv=u_i v_i\ ,\qquad
	\displaystyle{\|\bu\|=(\bu\cdot \bu)^{\frac{1}{2}}}\qquad
	\forall \,\bu, \bv\in \mathbb{R}^d, \\
	&\bsigma:\bgamma=\sigma_{ij}\gamma_{ij}\ ,\qquad
	\|\bsigma\|=(\bsigma:\bsigma)^{\frac{1}{2}} \qquad\,\forall\,
	\bsigma,\bgamma\in\mathbb{M}^d.
\end{eqnarray*}
\noindent Let $\bu$ and $\boldsymbol{\Pi}$ denote the displacement field and the first {Piola--Kirchhoff} stress tensor, respectively, 
and the normal and tangential components of $\bu$ on $\Gamma$, which are given by $u_{n} = \bu\cdot{\bn}$, $\bu_{\tau}=\bu - u_{n}\bn$, where $\bn$ is the unit normal outside $\Gamma$. We consider that an index following a comma represents the partial derivative with respect to the corresponding spatial variable of $\bx$,\ $\displaystyle u_{i,j}=\frac{\partial u_i}{ \partial x_j}$. Dots above a function represent partial derivatives with respect to time, i.e.~$\dot\bu=\displaystyle\frac{\partial\bu}{\partial t}$ and
$\ddot\bu=\displaystyle\frac{\partial^2\bu}{\partial t^2}$.
Moreover, we recall that the divergence operator acting on a tensor field $\boldsymbol\tau$ is ${\rm Div}\,\boldsymbol{\tau}=(\tau_{ij,j})$.\\
In the problems studied later, the behavior of the material is described by a hyperelastic constitutive law. We recall that hyperelastic constitutive laws are characterized by the first Piola--Kirchhoff tensor $\boldsymbol{\Pi}$, which derives from a deformation energy density $W: \Omega\times\mathbb{ M}^d_+\to\mathbb{R}$, $\boldsymbol{\Pi}(\bx,\mathbf F)=
\frac{\partial}{\partial {\bf F}} W(\bx,{\bf F})=\partial_{\bf F} W(x,{\bf F})$, for all $\bx \in\Omega$ and ${\bf F}\in\mathbb M_+^d$, where $\mathbb M_+^d = \{{\bf F}\in \mathbb M^d: \det {\bf F} > 0\}$. Here ${\bf F}$ is the deformation gradient defined by ${\bf F} = {\bf I} + \nabla {\bu}$
and $\partial_{\bf F}$ represents the differential with respect to the variable ${\bf F}$, for more details on hyperelasticity see \cite{ciarlet1988mathematical,le1994numerical,laur2002}. In what follows, we consider a dynamic problem with contact and friction in which the hyperelastic body comes into contact with a perfectly rigid obstacle $\Omega^{obs}\subset \mathbb{R}^d$ $(d=1,2,3)$ which is a rigid bounded domain  with a continuous Lipschitz boundary $\Gamma^{obs}$, the so-called foundation (see Figure \ref{body}).
\begin{figure}[h!]
	\centering
	\subfloat[]{\includegraphics[scale=.425]{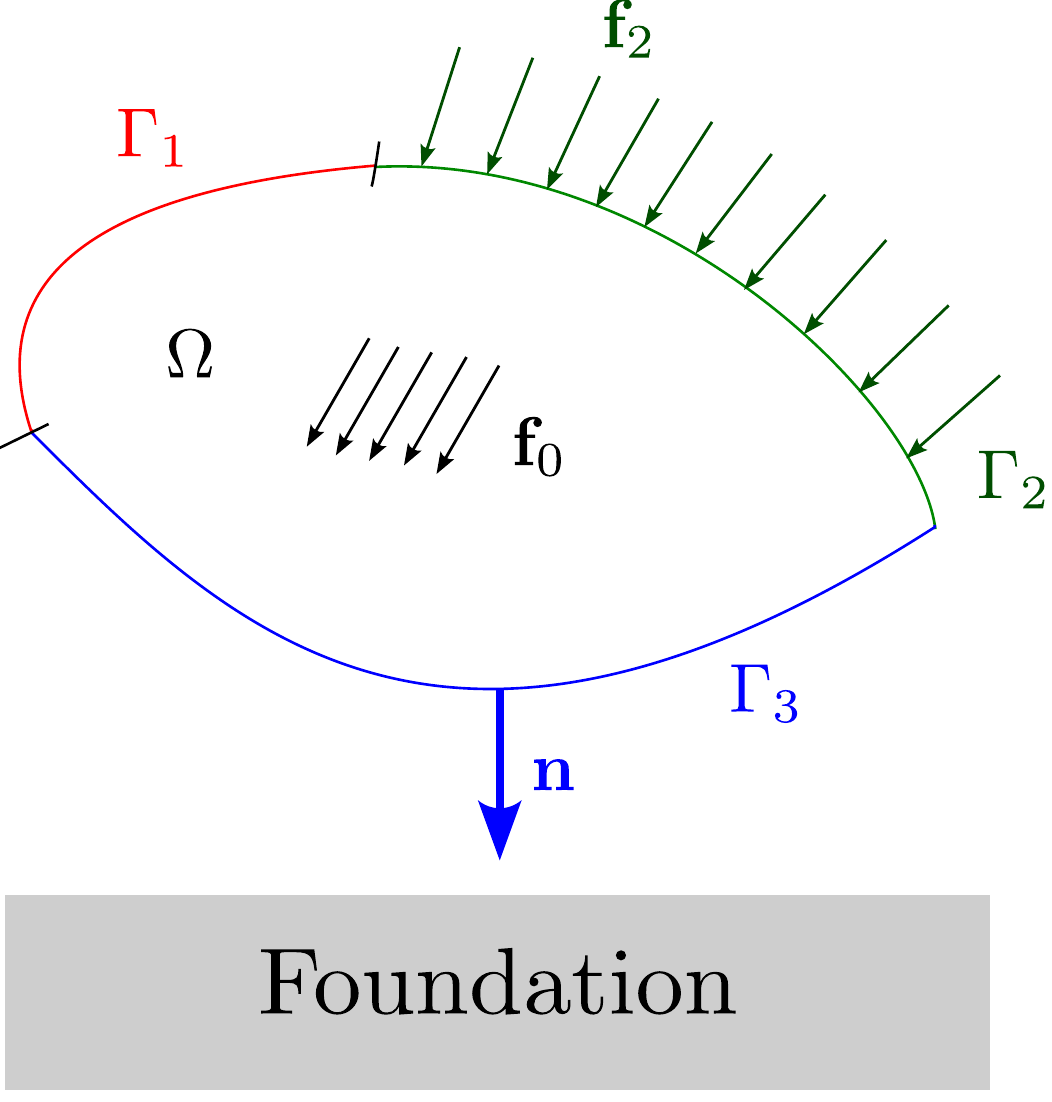}} \hspace{2cm}
	\subfloat[]{\raisebox{.05cm}{\includegraphics[scale=.425]{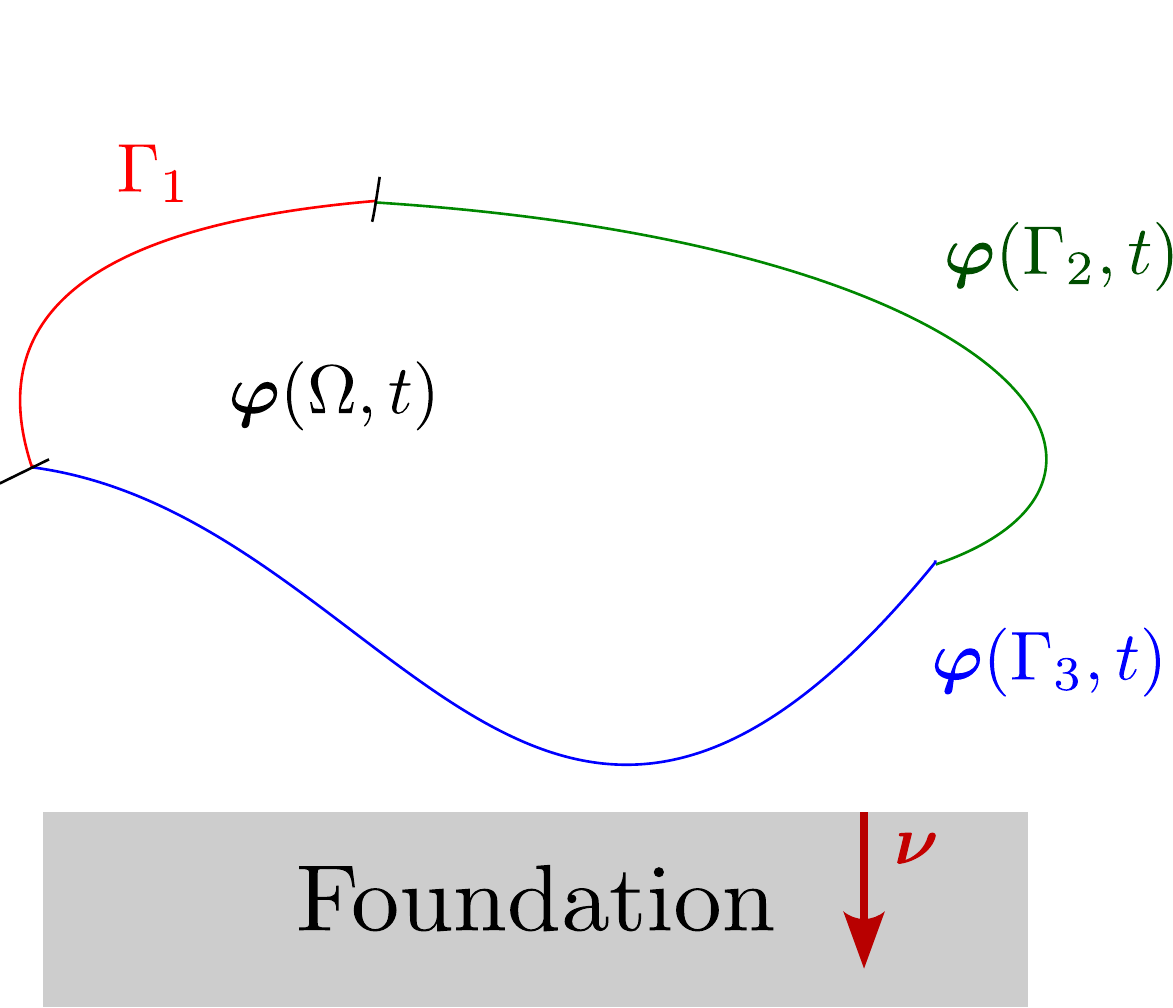}}}
	\caption{Reference (a) and deformed (b) configurations of a body.}
	\label{body}
\end{figure}
The hyperelastic body is subjected to the action of volumetric forces with density $\fb_0$ and surface tractions with density $\fb_2$
which act on $\Gamma_2$. In the rest of the paper, we consider the time interval of interest $[0,T]$ with $T>0$. We denote by $t\in [0,T]$ the time variable.

\subsection{Frictional contact conditions}\label{reg_cd}

Suppose the body is fixed on $\Gamma_1$ and can come into contact at $\Gamma_3$ with the foundation. In the following, the frictional contact conditions are based on the combination of normal compliance conditions with a Coulomb law of dry friction on $\Gamma_3$. We denote by $\bvarphi:\overline{\Omega}\times[0,T]\to\mathbb{R}^d$ the deformation field, with $\bvarphi(\bx,t)=\bx+\bu( \bx,t)$ the position at time $t\in[0,T]$ of point $\bx\in\overline{\Omega}$. For any point $\bx\in\Gamma_3$, we define the point $\overline{\by}(\bx,t)$ of the foundation closest to $\bx$:
\begin{eqnarray*}
	\overline{\by}(\bx,t)=\arg\min_{{\by \in \Gamma^{obs}\color{black}}}{\|\bvarphi (\bx,t)-\by \|_2}.
\end{eqnarray*}
In this way, one can define the minimal allowed contact distance (gap) between a point of $\Gamma_3$ and its orthogonal projection on the rigid foundation as follows:
\begin{eqnarray*}
	g_{\nu}=(\bvarphi(\bx,t)-\overline{\by}(\bx,t))\cdot\bnu, \quad \forall \bx \in \Gamma_3,
\end{eqnarray*}
\noindent where $\bnu$ is the inner unit normal vector to the rigid foundation. The normal force of contact $\Pi_\nu$, assumed to be negative, can be written in the direction $\bnu$:
\begin{eqnarray*}
	\Pi_\nu=\bnu \cdot \bPi \bn.
\end{eqnarray*}
In the same way, the tangential force of contact can also be expressed according to the first tensor of Piola--Kirchhoff:
\begin{eqnarray*}
	\bPi_{\tau}=\bPi\bn-\Pi_{\nu} \bnu,
\end{eqnarray*}
With these definitions in place, the tangential contact velocity $\dot{\bg}_\tau$ of a point $\bx\in \Gamma_3$, relative to the opposite surface of the foundation, is given by
\begin{eqnarray*}
	\dot{\bg}_\tau=[\bI_d-\bnu \otimes \bnu]\dot{\bu} (\bx,t),
\end{eqnarray*}
The conditions of contact and friction are posed on the boundary $\Gamma_3$ or $\Gamma_3\times[0,T]$ in a time-dependent problem and pertain to the normal and tangential components $\Pi_\nu$ and $\bPi_\tau $ of the surface contact force, respectively, and the displacements $u_\nu$, $\bu_\tau$. With these considerations, the unilateral conditions on the boundary of contact $\Gamma_3$ are given in the following section.
\subsubsection{Law of unilateral contact with friction}
The law of unilateral contact of a solid on a rigid obstacle was proposed in 1933 by Signorini \cite{Signorini} and is written in the form of three conditions: a condition of non-penetration, a condition of compression and a condition of complementarity.  Thereafter and for convenience, we will change notation for the contact variables:
$\delta_{\nu}=g_{\nu}$ and $\lambda_{\nu}=-\Pi_{\nu} $. Therefore, the unilateral contact conditions read
\begin{equation}\label{cont_uni}
\begin{aligned}
		&\textrm{Non-penetration condition:}\quad&\delta_{\nu}\leq 0,  \\
		&\textrm{Compression condition:} \quad&\lambda_\nu\geq0, \\
		& \textrm{Complementarity condition:} \quad&\lambda_\nu\delta_\nu=0. \\
\end{aligned}
\end{equation}
From a mechanical viewpoint, this amounts to considering a perfectly rigid foundation; no matter how much compressive force is applied, no penetration occurs.\\
Coulomb's law of friction involves the tangential friction stress $\bPi_\tau $, the normal contact pressure $\Pi_\nu$, and the tangential contact velocity $ \dot \bu_{\tau}$ that we now denote $\dot\bdelta_{\tau}=\dot\bu_{\tau}$, as well as the tangential stress $\blambda_{\tau}=-\bPi_{\tau}$, as follows:
\begin{equation}
	\begin{aligned}
		&\|\blambda_\tau\|\le\mu\,|\lambda_\nu|&\quad \rm stick\ status, &\\
		&\displaystyle\|\blambda_\tau\|=\mu\,|\lambda_\nu|\,\frac{{\dot\bdelta}_\tau}{\|{\dot\bdelta}_\tau\| } \ \ {\rm if}\ \ \dot{\bdelta}_\tau\ne 0
		&\quad \rm slip\ status.
	\end{aligned}
\end{equation}
\noindent where $\mu \ge 0$ is the coefficient of friction.\\
If the norm of the tangential stress $\blambda_\tau$ is less than the friction threshold $\mu\,|\lambda_\nu|$, then there is sticking between the body and the foundation. If, on the other hand, this threshold is reached, then the body slides on the foundation while the tangential stress is constant and depends on the unit tangent $\btau={\dot\bdelta}_{\tau}/\|{\dot\bdelta}_{\tau}\|$.

\noindent Note that it is possible to write conditions (\ref{cont_uni}) as the following subdifferential inclusion:
\begin{equation}
	\lambda_{\nu} \in \partial \Psi_{\mathbb{R}^{-}}(\delta_{\nu})\quad {\rm on}\quad \Gamma_3\times(0,T) ,
\end{equation}
where $\partial$ represents the sub-differential operator in the sense of convex analysis and $\Psi_A$ denotes the indicator function of the set $A \subset \mathbb{R}$.
A similar consideration for frictional stress leads to
\begin{equation}
	\blambda_{\tau}\in \mu\lambda_{\nu}\partial\|{\dot\bdelta}_{\tau}\|
	\quad{\rm on}\quad\Gamma_3\times(0,T).
\end{equation}

\subsubsection{Law of persistent contact with friction}
Such a law is slightly on the margins of the previous one insofar as its main interest resides in its natural properties of energy conservation. The persistency condition is expressed as a complementarity condition between the normal stress $\lambda_\nu$ and the tangential velocity $\dot\delta_\nu$, namely:
\begin{eqnarray}
	\lambda_\nu \dot{\delta}_\nu={0}. \label{pers}
\end{eqnarray}
This condition alone is sufficient to guarantee that the work of the contact normal force $\displaystyle{\int_{\Gamma_3}\lambda_\nu \dot{\delta}_\nu}=0$ vanishes. By combining with the unilateral contact law, it takes the following form:
\begin{equation}
	\begin{cases}
		\rm{if} \quad \delta_\nu<0,\quad\lambda_\nu=0,&\\
		\rm{if} \quad \delta_\nu=0, \quad\lambda_\nu\in-\partial\Psi_\mathbb{R^+}(\dot{\delta}_\nu).
	\end{cases}
\end{equation}
We refer the reader to~\cite{laur2002,hauret2006energy,ayyad2009formulation} for more details.
\subsubsection{Normal compliance law via a $\alpha$-Moreau--Yosida regularization}
We briefly present the $\alpha$-Moreau--Yosida regularization \cite{conj, MoYosi} of the unilateral condition of Signorini and begin with a reminder of the concepts of variational analysis. Let $S$ be a subset of a Hilbert space $\mathbf{H}$ endowed with the norm $\|{\cdot}\|$.

\begin{definition}
	Let $f$ be a lower bounded semicontinuous function defined by 
	$f: \mathbf{H}  \to  \mathbb{R}\cup {\{+\infty\}}$.
	For all $r>0$, the $\alpha$-Moreau--Yosida envelope \cite{conj, MoYosi, reguEmi, regular} of $f$, with $\alpha\geq 2$, is defined by:
	\begin{equation}
		f_r^{\alpha}(z)=\inf_{y\in \mathbf{H}}\displaystyle{\Big(f(y)+\frac{1}{r}\|y-z\|^\alpha\Big),\quad\forall z\in\mathbf H. }
	\end{equation}
\end{definition}
\begin{theorem}
	Let $f: \mathbf{H} \to  \mathbb{R}\cup {\{-\infty\}}$, the regularization of $\partial f$ for all $r>0$ is the gradient $\nabla{f_r^{\alpha}}$ associated with the envelope $f_r^{\alpha}$.
\end{theorem}
\noindent  For the details of proof, see  \cite{conj, MoYosi,reguEmi,regular}\color{black}.\\
If we take $f=\Psi_{\mathbb{R}^-}$, for $r>0$ and $\alpha\geq 2$, we have
\begin{equation*}
	(\Psi_{\mathbb{R}^-})_r^{\alpha}(z)=\inf_{y\in \bf H}\displaystyle{\Big(\Psi_{\mathbb{R}^-}(y )+\frac{1}{r}\|y-z\|^\alpha\Big)}
	=\inf_{y\in \mathbb{R}^-}\Big(\displaystyle{\frac{1}{r}\|y-z\|^\alpha\Big)} \\
	\eqqcolon\displaystyle{\frac{1}{r} {\rm {\rm dist}}_{\mathbb{R}^-}^\alpha}(z).
\end{equation*}
\begin{proposition}
	Let $\alpha\geq2$ on;
	\begin{equation}
		\nabla(\Psi_{\mathbb{R}^-})_r^{\alpha}(z)=\nabla \displaystyle{\Big(\frac{1}{r} {\rm dist}_{\mathbb{R} ^-}^\alpha}(z)\Big) \\
		=\frac{1}{r}{\rm proj}_{\mathbb{R}^+}(z)\|{\rm proj}_{\mathbb{R}^+}(z)\|^{\alpha-2}.
	\end{equation}
\end{proposition}

The contact with a deformable foundation is modeled by the normal compliance condition. It attributes a reactive normal pressure depending on the interpenetration of the foundation; this means that the normal stress $\lambda_\nu$ is a function of the normal displacement $\delta_\nu$. A general expression for the normal compliance condition is then given by
\begin{equation}
	\lambda_\nu=p(\delta_\nu) \quad\mbox{on} \quad \Gamma_3\times(0,T)
\end{equation}
where $p(\cdot)$ vanishes for negative arguments.
The normal compliance condition was first introduced in \cite{MartOden,OM}. This standard condition can be considered as a Moreau--Yosida regularization of the unilateral contact condition of Signorini with $\alpha = 2$ and $z=\delta_\nu$: \\
{\it  Standard Normal Compliance (SNC) with friction}
\begin{equation}\label{compl}
	\lambda_\nu=c_\nu\frac{\alpha}{2}([\delta_\nu]_+)^{\alpha-1}\quad \mbox{on} \quad \Gamma_3\times(0, T),
\end{equation}
where $[x]_+$ is the positive part of $x\in \mathbb R$, and $c_\nu = \frac{1}{r} $ can be assimilated to the stiffness coefficient of the foundation. This law of normal compliance (\ref{compl}) has two particularities: first, it allows to reduce the interpenetration while the second aspect lies in its natural quasi-energy-conservation properties (see energy balance in section \ref{prop_ener}), which is critical from a physical point of view. In adequacy with this regularization process, it is assumed that the law of friction can expressed as
\begin{eqnarray}\label{compl1}
	\left\{\begin{array}{ll}
		\textrm{if}\quad ||\blambda_\tau||<\mu |\lambda_\nu|\quad \blambda_\tau=c_\tau\dot\bdelta_\tau, \ \rm stick &\\
		\textrm{if}\quad ||\blambda_\tau||=\mu |\lambda_\nu|\quad \blambda_\tau=\mu \lambda_\nu\displaystyle{\frac{\dot\bdelta_\tau} {\|\dot\bdelta_\tau\|}}, \ \rm slip
	\end{array}\right.
	\mbox{on}&\ \Gamma_3\times(0,T),
\end{eqnarray}
where $c_\tau > 0$ is the tangential compliance parameter.

\subsection{Strong formulation of the mechanical problem}\label{strong}
With the preceding notation, the strong formulation of the problem  is the following one.\\
\medskip\noindent
{\bf Problem} ${\cal P}$. {\it Find displacement field
	$\bu:\Omega\times [0,T]\to\mathbb{R}^d$ and the stress field
	$\boldsymbol{\Pi}:\Omega\times [0,T]\to\mathbb{M}^d$ such that}
\begin{eqnarray}
	\label{1} \boldsymbol{\Pi}=\partial_{\bf F} W({\bf F})\quad&{\rm in}\
	&\Omega\times(0,T),\\[3mm]
	\label{2} {\rm Div}\,\boldsymbol{\Pi}+\fb_0= \rho\ddot{\bu}\quad&{\rm
		in }\ &\Omega\times(0,T),\\[3mm]
	\label{3} \bu=\bzero\quad &{\rm on}\ &\Gamma_1\times(0,T),\\[3mm]
	\label{4} \boldsymbol{\Pi}\bnu=\fb_2\quad&{\rm on}\
	&\Gamma_2\times(0,T),\\[3mm]
	\label{5} \lambda_\nu=c_\nu\frac{\alpha}{2}[\delta_\nu]_+^{\alpha-1} \quad \quad&{\rm on}\
	&\Gamma_3\times(0,T),\\[3mm]\label{6} \left\{\begin{array}{ll} ||\blambda_\tau||<\mu|\lambda_\nu| \quad \blambda_\tau=c_\tau\dot\bdelta_\tau & \\
		||\blambda_\tau||=\mu|\lambda_\nu|\quad \blambda_\tau=\mu \lambda_\nu\displaystyle{\frac{\dot\bdelta_\tau}{\|\dot\bdelta_ \tau\|}}\end{array}\right. \quad
	&{\rm on}&\ \Gamma_3\times(0,T), \\[3mm]
	\label{7}\bu(0)=\bu_0,\ \dot{\bu}(0)=\bu_1 \quad &{\rm in}&\ \Omega.
\end{eqnarray}
with $\alpha\ge 2$, $\mu$ the coefficient of friction depending on the sliding rate  and  $c_\nu$, $c_\tau$ the compliance parameters.\\
Equation (\ref{1}) represents the hyperelastic constitutive law of the material, (\ref{2}) represents the equation of motion in which $\rho> 0 $ is the material density and is assumed constant, for simplicity. The conditions (\ref{3}), (\ref{4}) represent respectively the boundary conditions of displacement and traction. Conditions (\ref{5}) and (\ref{6}) respectively represent the conditions of contact with normal compliance and friction described in the preceding section. Finally, (\ref{7}) represents the initial conditions in which $\bu_0$ and $\bu_1$ are respectively the initial displacement and velocity.

\subsection{Variational formulation of the problem}\label{variat}

In order to derive the variational formulation of {\bf Problem} ${\cal P}$, additional notation and some preliminary elements are necessary. The classical notation for the Sobolev and Lebesgue spaces associated with $\Omega$ and $\Gamma$ is used.
We consider a closed subspace of $H^{1}$ as follows:
Following the approach of Duvaut and Lions \cite{duvaut_lions}, we introduce the following Hilbert spaces:
$$ V=\{\bv\in H^{1}(\Omega;\mathbb{R}^d): \bv=\mathbf 0\,\,\,\text{on} \,\,\,\Gamma_1\} \quad  \text{and} \quad H=L^{2}(\Omega;\mathbb{R}^d)$$
they are Hilbert spaces endowed with the scalar products $(\bu,\bv)_V$ and $(\boldsymbol{\Pi},\betau)_H$ and their associated norms $\|{\cdot}\|_ {V}$ and $\|{\cdot}\|_{H}$ respectively.\\
Note that the Lagrange multipliers $\lambda_{\nu}$ and $\blambda_{\tau}$ are taken equal to $-\Pi_{\nu}$ and $-\bPi_{\tau}$, respectively.\ \
In order to introduce a variational formulation of the mechanical problem, we consider the following spaces \cite{duvaut_lions,khenous2006hybrid}:
$$X = \left\{\,v|_{\Gamma_3}: \ \bv \in V\, \right\} \subset H^{1/2}(\Gamma_3;\mathbb{R}^d)$$
$$X_\nu=\left\{\,v_\nu|_{\Gamma_3}: \ \bv \in V\, \right\} ,  \ \ X_\tau=\left\{\,\bv_\tau|_{\Gamma_3}: \ \bv \in V\, \right\}$$
and their topological dual spaces  $X^{'}$, $X^{'}_\nu$ and $X'_\tau$. We assume that $\Gamma_3$ is sufficiently regular so that $X_\nu \subset H^{1/2}(\Gamma_3;\mathbb{R})$, $X_\tau \subset H^{1/2}(\Gamma_3;\mathbb{R}^{d-1})$, $X^{'}_\nu \subset H^{-1/2}(\Gamma_3;\mathbb{R})$ and $X^{'}_\tau \subset H^{-1/2}(\Gamma_3;\mathbb{R}^{d-1})$.
Moreover, we denote by
$\langle\cdot,\cdot\rangle_{X_{\nu}^{'},X_{\nu}}$ and $\langle\cdot,\cdot\rangle_{X_{\tau}^{'}, X_{\tau}}$ the corresponding duality products. For more details on the operators traces, we can refer to \cite{khenous2006discretization,khenous2006hybrid}.

We move on to the variational formulation (or weak formulation) for {\bf Problem} ${\cal P}$.
Multiplying $(\ref{2})$ by any virtual velocity $\bv$ and applying Green's formula, we get
\begin{equation}\label{G_step2}
	- \int_{\Omega} {\bf \Pi}(t):\nabla{\bv}\, dx + \int_{\partial \Omega} {\bf \Pi}(t){\bn}\cdot{\bv}\ da + \int_{\Omega} {\boldsymbol f}_0(t)\cdot{\boldsymbol v} \,dx = \int_{ \Omega} \rho \ddot{\boldsymbol{u}}(t)\cdot{\bv}\, dx
\qquad\forall\bv\in V. \color{black}\\
\end{equation}
To establish the variational formulation of {\bf Problem} ${\cal P}$ (\ref{1})--(\ref{7}), we need some additional notation.
Thus, we consider the function $\fb:(0,T)\rightarrow V^*$ where the exterior volume and surface force densities are assumed to be such that
\begin{eqnarray}
	\label{f}
	&& \fb_0 \in L^2(0,T;L^{2}(\Omega)), \qquad \fb_2\in L^2(0,T;L^{2}(\Gamma_2)),
\end{eqnarray}
so that
\begin{eqnarray}
	&&\label{ff}(\fb(t),\bv)_V=(\fb_0(t),\bv)_H+(\fb_2(t),\bv)_{L^2(\Gamma_2)},\qquad\forall\bv\in V .
\end{eqnarray}
Using the duality between $V^*$ and $V$, Green's formula and the contact conditions with friction on the boundary $\Gamma_3$, the equation (\ref{G_step2}) becomes:
\begin{eqnarray}
	&& \langle{\rho\ddot{\bu}(t),\bv}\rangle_{{V^*}\times{V}} + \label{x1}\langle{\boldsymbol{\Pi}(t),\nabla\bv}\rangle_{{V^*}\times{V}}=(\fb(t),\bv)_V
	+\int_{\Gamma_3}\,\Pi_\nu(t)
	v_\nu\,da+\int_{\Gamma_3}\,\boldsymbol{\Pi}_\tau(t)\cdot\bv_\tau\,da.\nonumber
\end{eqnarray}
Finally, we obtain the variational formulation of the contact problem with friction ${\cal P}$ using the Lagrange multiplier $\lambda_{\nu}$, related to the normal contact stress $\Pi_{\nu}$, and the Lagrange multiplier $\blambda_{\tau}$, related to the tangential contact stress $\bPi_{\tau}$, in terms of two unknown fields (here and in the following, we drop explicit mention of time dependence for ease of presentation).\\
{\textbf{Problem} ${\cal P}_V$}:
{\it Find the displacement field $\bu \in L^\infty(0,T;V)$, with $\dot\bu \in L^2(0,T;V)$ and \mbox{$\ddot \bu \in L^2(0,T;V^*)$}, the normal stress field $\lambda_{\nu}: (0,T)\rightarrow X_{\nu}^{'}$ and the tangential stress field $\blambda_{\tau}: (0,T)\rightarrow X_{\tau}^{'}$ such that, $\forall\,\bv\in V$},
\begin{eqnarray}
	\label{10}
	\langle\rho\ddot{\bu},\bv\rangle_{{V^*}\times{V}}+\langle{\boldsymbol{\Pi},\nabla\bv}\rangle_{{V^* }\times{V}} + \langle \lambda_{\nu},v_{\nu}\rangle_{X_{\nu}^{'},X_{\nu}}+\langle\blambda_{\tau}, \bv_{\tau}\rangle_{X_{\tau}^{'},X_{\tau}}=
	(\fb,\bv)_V, \quad \text{in }\ \ (0,T),\\[3mm]
	\label{11}\lambda_\nu=c_\nu\frac{\alpha}{2}[\delta_\nu]_+^{\alpha-1} \quad {\rm on}\quad\Gamma_3\times (0,T),\\[3mm]
	\label{12}
	\left\{\begin{array}{ll} ||\blambda_\tau||<\mu|\lambda_\nu|\quad \blambda_\tau=c_\tau\dot\bdelta_\tau &\\
		||\blambda_\tau||=\mu|\lambda_\nu|\quad \blambda_\tau=\mu \lambda_\nu\displaystyle{\frac{\dot\bdelta_\tau}{\|\dot\bdelta_ \tau\|}}\end{array}\right.\quad {\rm on}\quad\Gamma_3\times(0,T),\\[3mm]\nonumber
\end{eqnarray}
{\it and, moreover,}
\begin{equation}\label{BDSxx}
\bu(0)=\bu_0,\qquad\dot\bu(0)=\bu_1.
\end{equation}
According to Hadamard's definition of the meaning of well-posed problems, it is known that the variational problem ${\mathcal P}_V$ is ill-posed. Indeed, this problem is very complex due to the multiplicity of its solutions for several reasons: hyperelastic medium, dynamic process, and friction conditions.  For further details on the existence of solution in particular hyperelastic cases and with more regularity, one can refer to the following references: \cite{ciarlet1988mathematical,le1994numerical,barboteu2018analysis}.

\subsection{Energy conservation properties in the continuous case}\label{prop_ener}
From a physical point of view, the solution of a hyperelastic problem should satisfy some conservation properties like conservation of energy, conservation of kinematic momentum, and conservation of linear momentum.
We are now particularly interested in the energy conservation properties.
In the absence of contact and friction, the conservation of energy can be written as follows:
\begin{eqnarray}
	\int_{0}^{t} \dot{E}(s)ds = E(t)
	-E(0)
	= \int_{0}^{t}\int_{\Omega} {\bf f}_0\cdot \dot{\bu}\,\,dx \, ds \ + \ \int_{0}^{t}
	\int_{\Gamma_2} {\bf f}_2\cdot \dot{\bf u}\,\, da\, ds,
\end{eqnarray}
where $(s,t) \in (0,T)\times(0,T)$ and ${E}(t)$ denotes the internal energy of the system at time $t$, defined as
\begin{equation}
	{E}(t) = \frac{1}{2} \int_{\Omega} \rho |\dot{\bf u}|^{2} dx +
	\int_{\Omega} \widetilde{W}({\bf C})dx \label{ene_elas}
\end{equation}
where ${\bf C} = {\bf F}^T {\bf F}$ is the Cauchy-Green tensor and $\widetilde{W}({\bf C}) := W({\bf F})$ \color{black}. \\
Moreover, the conservation of energy for hyperelastic phenomena with frictional contact is written in the following way:
\begin{equation}
\begin{aligned}
	E(t)-E(0)
	= \int_{0}^{t}\int_{\Omega} {\bf f}_0\cdot {\dot\bu}\,\, dx  ds \ + \ \int_{0}^{t}
	\int_{\Gamma_2} {\bf f}_2\cdot {\dot\bu}\,\, da ds \ \\ -\int_{0}^{t}\int_{\Gamma_3}\,\lambda_\nu
	\dot\delta_\nu\,da ds-\int_{0}^{t}\int_{\Gamma_3}\,\blambda_\tau \cdot \dot\bdelta_\tau\,da ds 
\end{aligned}
\label{bil_ini_1}
\end{equation}
Using the variational formulation of the problem by taking ${\bv} = \dot{\bu}(t,{\bx})$, we obtain the frictional contact reaction work given by
\begin{equation}
	\mathcal{W}_{c+f} =
	\int_{\Gamma_3} ({\lambda_\nu}
	{\dot\delta_\nu}+{\blambda_\tau \cdot \dot\bdelta_\tau}
	)\ da.
\end{equation}
where $ \dot\delta_\nu$ and $\dot\bdelta_\tau$ represent the time derivatives of $\delta_\nu$ and ${\bdelta_\tau}$ respectively. \\
When the volume and surface forces are not taken into account, and with the presence of the normal contact compliance law without friction, the expressions of the energy allow us to obtain
\begin{equation}
	\displaystyle{E(t)-E(0)
		=-\int_{0}^{t}\int_{\Gamma_3} {\lambda_\nu}
		{\dot\delta_\nu}}\ da\ ds.
\end{equation}\label{cons0}
Using the general normal compliance condition (\ref{11}), we get:
\begin{equation}\label{cons}
	\displaystyle{E(t)-E(0)=-
		\int_{0}^{t}\int_{\Gamma_3}c_\nu\frac{\alpha}{2} ( \delta_\nu\big]_{+}^{\alpha-1}{\dot\delta_\nu} da \ ds}
	=-\frac{c_\nu}{2}\int_{\Gamma_3}\big(\big[\delta_\nu\big]_{+}^{\alpha}(t)-\big[\delta_\nu\big]_{+}^{\alpha}(0)\big)\, da
\end{equation}
The difference $\big[\delta_\nu\big]_{+}^{\alpha}(t)-\big[\delta_\nu\big]_{+}^{\alpha}(0)$ in (\ref{cons}) is very small since the penetrations $\big[\delta_\nu\big]_{+}(t)$ at any time $t$ are also small as long as the compliance parameter $c_\nu$ is sufficiently large. Therefore, the energy of the system is ``almost" conserved: $E(t)\approx E(0)$.\\
 Considering now the friction, we obtain
\begin{equation}\label{diss}
	{\blambda_\tau}\cdot{{\dot
			\bdelta_\tau}} \geq 0 \Rightarrow
	- \mathcal{W}_{c+f} \leq 0 \quad \Rightarrow \ E(0) \geq E(t)
\end{equation}
Concerning the expression (\ref{diss}), we note a dissipation of energy between the instants $0$ and $t$ because of the friction, which is physically acceptable due to the dissipative nature of this phenomenon.\\
When we take these energy expressions with the ``persistent'' conditions (\ref{pers}), we get the following results:
	\begin{eqnarray}\label{pe}
		&{\rm{Case\ without\ friction:}}\quad {\lambda_\nu}{{\dot
				\delta_\nu}}=0,\ \blambda_\tau\cdot\dot{\bdelta_\tau} = 0  \Rightarrow
		\mathcal{W}_{c+f} = 0 \Rightarrow \ E(0) = E(t), \\
			&	{\rm{Case\ with\ friction:}}\ \label{di} \ {\lambda_\nu}{{\dot
				\delta_\nu}}=0,\ {\blambda_\tau}\cdot{{\dot\bdelta_\tau}}\geq 0 \Rightarrow
		\mathcal{W}_{c+f}\geq 0   \Rightarrow \ E(0) \geq E(t).
	\end{eqnarray}
Here, (\ref{pe}) expresses the conservation of total energy when the persistence condition is applied, whereas by (\ref{di}), because of friction, we observe a dissipation of energy between the instants $0$ and~$t$.

\section{ Discrete formulation of the frictional contact problems}\label{s3}

\subsection{Variational approximation}\label{variat_approx} 
In this section, we introduce a discrete approximation in time and space of the problem ${{\mathcal P}}_V$, based on arguments similar to those used in \cite{ayyad2009formulation, ayyad2009frictionless, barboteu2015hyperelastic, barboteu2015dynamic, barboteu2018analysis, khenous2006discretization, kpr}. First of all, we recall some preliminary elements concerning the time discretization step.\\
Let $ N $ be an integer, and $\Delta t=\frac{T}{N}$ a time step. 
For a continuous function $f$ with respect to time, we will use the notation $f_j=f(t_j)$ for $0 \leq j \leq N$. In what follows, we consider a collection of discrete times $\{t_n\}_{n=0}^{N}$ which defines a uniform partition of the time interval $[0,T]= \bigcup_{ \scriptstyle n=1}^{N}[t_{n-1},t_{n}]$, with $t_0=0$, $t_{n}=t_{n-1} + \Delta t$, and $t_N=T$. Using truncated Taylor expansions, we find 
\begin{equation}
	 \bu_{n}=
	\bu_{n-1}+\frac{\Delta t}{2}(\dot\bu_{n}+\dot\bu_{n-1}) + \theta(\Delta t^2),
\end{equation}
Finally, for a finite sequence $\{u_n\}_{n=1}^N$, we denote the divided differences of the midpoints by:
\begin{equation}\label{midp}
	\dot \bu_{n-\frac{1}{2}}=\frac{\bu_{n}-\bu_{n-1}}{\Delta t} = \frac{1}{2}(\dot \bu_{n} + \dot \bu_{n-1}),
\end{equation}
 Then we use the notation $\Box_{n-\frac{1}{2}} = \frac{1}{2}(\Box_n +
\Box_{n-1})$, where $\Box_n$ represents the approximation of
$\Box(t_n)$. For regularity in time of the solution, see \cite{hht,laur2002,wr}.  \color{black} Let us note that the time integration scheme employed is based on an implicit scheme of order 2 which one finds in (\ref{midp}).\\
We now present some elements concerning the spatial discretization step. Let $\Omega$ be a polyhedral domain. Consider a regular partition $\{{\cal T}^h\}$ of triangular finite elements of $\overline {\Omega}$ which are compatible with the boundary decomposition
$\Gamma=\overline{\Gamma_1}\cup\overline{\Gamma_2}\cup\overline{\Gamma_3}$, i.e., if one side of an element $\mathrm T\in{\cal T}^h$ has more than one point on $ \Gamma $, then the side lies entirely on $\overline{\Gamma_1}$, $\overline{\Gamma_2}$ or
$\overline{\Gamma_3}$. The space $V$ is approximated by the finite dimensional space $V^h \subset V$ of continuous and piecewise affine functions, that is:
\begin{eqnarray}
	&& \hspace*{-0.7cm}V^h=\{\,\bv^h\in [C(\overline{\Omega})]^d \;:
	\; \bv^h|_{\mathrm T}\in [P_1(\mathrm T)]^d\,
	\,\,\, \forall\,\mathrm T\in {\cal T}^h, \nonumber\\
	&& \qquad \qquad \quad \bv^h=\bzero \,\,\, \hbox{at nodes on}\,\,\, \Gamma_1\},\nonumber
\end{eqnarray}
where $P_1(\mathrm T)$ represents the space of polynomials of degree less than or equal to 1 in $\mathrm T$ and $h>0$ is the spatial discretization parameter. Let us consider the following spaces \cite{khenous2006discretization,khenous2006hybrid}:
$$X^h_\nu=\left\{\,v^h_\nu|_{\Gamma_3}: \ \bv^h \in V^h\, \right\} ,  \ \ X^h_\tau=\left\{\,\bv^h_\tau|_{\Gamma_3}: \ \bv^h \in V^h\, \right\}$$
and their topological dual spaces   $X^{'h}_\nu$ and $X^{'h}_\tau$. We assume that $X^{'h}_\nu \subset X^{'}_\nu \cap L^2(\Gamma_3;\mathbb{R})$ and $X^{'h}_\tau \subset X^{'}_\nu \cap L^2(\Gamma_3;\mathbb{R}^{d-1})$. \color{black} For the discretization of the Lagrange multiplier spaces $X_{\nu}^{h}$ and $X_{\tau}^{h}$, we use piecewise constant functions as done in \cite{Ac,pietrzak1999large,khenous2006discretization,barboteu2015hyperelastic, BDMP:21, convergence-contact}.
The discrete Lagrange multiplier spaces, denoted by $X_{\nu}^{'h}$ and $X_{\tau}^{'h}$, are related to the discretization of the
normal stress $\lambda_\nu$ and the discretization of the friction stress $\blambda_\tau$, respectively.\\
With these previous notation and the midpoint scheme (\ref{midp}), we have the following discrete approximation for the problem ${{\mathcal P}}_V$ at time $t_{n-\frac {1}{2}}$: \\
\medskip \noindent{\bf Problem} ${\mathcal P}_V^{h}$. {\it
	Find a discrete displacement field
	$\bu^{h}=\{\bu_n^{h}\}_{n=0}^N\subset V^h$, a discrete normal stress field $\lambda_{\ nu}^{h} =\{{\lambda_{\nu}}_{n}^{h}\}_{n=0}^N
	\subset X_{\nu}^{' h}$, and a discrete tangential stress field
	$\blambda_{\tau}^{h} =\{{\blambda_{\tau}}_{n}^{h}\}_{n=0}^N \subset
	X_{\nu}^{' h}$ such that, for all $n=1,\ldots,N$,}
\begin{eqnarray}
	&& \rho \ddot{\bu}^{h}_{n-\frac{1}{2}} + \bB(\bu^{h}_{n-\frac{1}{ 2}}) + {\lambda_{\nu}}^{h}_{n-\frac{1}{2}}\bnu_{n-\frac{1}{2}}+ {\blambda_ {\tau}}^{h}_{n-\frac{1}{2}}- \fb = \bzero \\[2mm]
	&&\label{LMch} {\lambda_{\nu}}^{h}_{n-\frac{1}{2}}=
	c_\nu\frac{\alpha}{2}[\delta^{h}_{\nu_{n-\frac{1}{2}}}]_+^{\alpha-1},\\ [2mm]
	&&\label{LMfh}
	\left\{\begin{array}{ll}
		{\|\blambda_{\tau}}^{h}_{n-\frac{1}{2}}\|<\mu{|\lambda_{\nu}}^{h }_{n-\frac{1}{2}}|\quad {\blambda_{\tau}}^{h}_{n-\frac{1}{2}}=c_{\tau } \dot\bdelta_{{\tau}^{h}_{n-\frac{1}{2}}} &\\
		{\|\blambda_{\tau}}^{h}_{n-\frac{1}{2}}\|=\mu|{\lambda_{\nu}}^{h }_{n-\frac{1}{2}}|\quad {\blambda_{\tau}}^{h}_{n-\frac{1}{2}}=\mu {\lambda_{\nu}}^{h}_{n-\frac{1}{2}}\displaystyle{\frac{\dot{\bdelta_{\tau}}^{h}_ {n-\frac{1}{2}}}{\|\dot{\bdelta_{\tau}}^{h}_{n-\frac{1}{2}}\|}} \end{array}\right.\quad\\[2mm]
	&&\label{cih12} \bu^{h}_{0}=\bar\bu^h_{0}, \quad\quad
	\dot\bu^{h}_{0}=\bar\bu^h_{1}.\nonumber
\end{eqnarray}
where $\bB(\cdot)$ is the internal stress operator representing the first Piola--Kirchhoff tensor ${\bPi}$. The initial values $\bar\bu^h_0 \in V^h$ and $\bar\bu^h_1 \in V^h$ are discrete values resulting from the finite element approximation of $\bu_0$ and $\bu_1$, respectively.

\subsection{Usual discrete framework of energy conservation}\label{stan_cd_dis}

In the rest of the section, to simplify notation and readability, we do not indicate the dependence of the different variables on the discretization parameter $h$, i.e., we write $\bu$ instead of $\bu^{h}$. We begin by recalling some preliminaries regarding discrete energy conservation in the hyperelastodynamic contactless framework.
In order to solve a hyperelastic dynamic problem, we have to use adapted time integration schemes. When considering nonlinear dynamical problems, standard implicit schemes ($\theta$-method, Newmark schemes, midpoint methods or HHT methods) lose their unconditional stability, as explained in \cite{hht, gonz, laur2002}.
Therefore, it is necessary to use implicit energy-conserving schemes like those used in \cite{arro,gonz,hauret2006energy,laur2002,simo,ayyad2009formulation,barboteu2015hyperelastic,Acary15} which are appropriate due to their long-term time integration accuracy and stability. In all these methods, the corresponding discrete mechanical conservation properties are satisfied.
By taking into account the implicit second-order temporal integration scheme of midpoint (\ref{midp}), we obtain the weak form of a nonlinear hyperelastodynamic problem integrated between times ${t_{n-1 }}$ and $t_{n}$:
\begin{eqnarray}\label{f_dyna_int}
	\begin{cases}
		\mbox{Find} \ {\bu}_{n} \in U \quad \mbox{such that}\\
		\displaystyle{\frac{1}{\Delta t} \int_{\Omega} \rho( \dot{\bu}_{n}-\dot{\bu}_{n-1})\cdot{\bv }\ \rm{dx}
			+ \int_{\Omega} {\bf \Pi}^{\text{algo}}_{n-\frac{1}{2}}:\nabla{\bv}\ \rm{dx} -
			\int_{\Gamma_2} { {\mathbf f}_2}_{n-\frac{1}{2}}\cdot{\bv}\ \rm{dx} -
			\int_{\Omega}
			{{\mathbf f}_0}_{n-\frac{1}{2}}\cdot{\bv}\ \rm{da}=0.}
	\end{cases}
\end{eqnarray}
There, the discrete tensor ${\bf \Pi }^{\text{algo}}$ is introduced in order to satisfy the exact properties of the discrete energy. This tensor, defined by Gonzalez in \cite{gonz}, takes the form:
\begin{equation}\label{syst}
	\begin{cases}
		{\bf \Pi }^{\text{algo}}_{n-\frac{1}{2}}=\textbf{F}_{n-\frac{1}{2}}{\bSigma}^{\text{algo}} \\
		\displaystyle
		{\bSigma}^{\text{algo}}=2\frac{\partial \widetilde W}{\partial \bC}(\bC_{n-\frac{1}{2}}) + 2[\widetilde{W}( {\bC}_{n})- \widetilde{W}({\bC}_{n-1})-\frac{\partial \widetilde W }{\partial \bC}(\bC_{n- \frac{1}{2}}):\Delta {\bC}_{n-1}]\frac{\Delta {\bC}_{n-1}}{\Delta {\bC}_{n -1}:\Delta {\bC}_{n-1}}
	\end{cases}
\end{equation}
with $\Delta {\bC}_{n-1}={\bC}_{n}- {\bC}_{n-1}$ and ${\bC}_{n-1}={\bF}^T_{n-1} {\bF}_{n-1}$. Using the arguments of \cite{gonz} and as shown by the axiom of material indifference which implies that $\widetilde{W}(\bF)=\widetilde{W}(\bC)$, it follows that (\ref{syst}) verifies the exact conservation of energy characterized by the following condition:
\begin{equation}\label{algo}
	\displaystyle{\bPi}^{\text{algo}}_{n-\frac{1}{2}}:({\bF}_{n}- {\bF}_{n-1})={\bPi} ^{\text{algo}}_{n-\frac{1}{2}}:(\nabla{\bu}_{n}-\nabla{\bu}_{n-1} )=\widetilde{W}( {\bC}_{n})- \widetilde{W}({\bC}_{n-1}).
\end{equation}
For more details on the standard energy conservation framework, we refer the reader to \cite{arro,gonz,hauret2006energy,laur2002,simo}.
Much work has been devoted to extending the conservative properties of the frictionless contact; more precisely, Lauren and Chawla \cite{Lach} and Armero and Petocz \cite{AP} showed the advantage of the persistence condition to conserve energy in the discrete setting. Nevertheless, in all these works, the numerical method shows that the interpenetration only disappears when the time step tends towards zero. In order to overcome this
drawback, Laursen and Love \cite{LL} have developed an efficient
method, by introducing a discrete jump in velocity; however, this
method requires the solution of an auxiliary system in order to
compute the velocity update results. Furthermore, Hauret and Le
Tallec \cite{hauret2006energy} have considered a specific penalized
enforcement of the contact conditions which allows to provide
energy conservation properties. Then, Khenous, Laborde and Renard
\cite{khenous2008mass} have introduced the Equivalent Mass Matrix method
(EMM), based on a procedure of redistribution of the mass matrix.
Interpretations and extensions of this method can be found in
\cite{H}. The resulting problem exhibits Lipschitz regularity in
time and achieves good energy evolution properties, due to the
fact that the persistency condition is automatically satisfied.
The EMM approach was studied and used in many
works; for instance, theoretical and computational aspects related
to this model can be found in \cite{HSWW,khenous2008mass}.

\subsection{Improved approach to ``almost'' conserve the energy}\label{imp_cd_dis}

In what follows, based on the papers \cite{hauret2006energy,barboteu2015hyperelastic}, we present a improved energy conservation method for hyperelastodynamic contact problems with its extension to dissipation phenomena with friction. This method allows to enforce the general normal compliance law during each time step with ``minimal'' contact penetrations and with conservation properties which respect ``almost'' the energy.\\
In order to take into account  contact at time $t_{n -\frac{1}{2}}$, we choose to implicitly approach the term of contact with friction; thus the weak formulation integrated between the times $t_{n-1}$ and $t_{n}$ is
\begin{eqnarray}\label{f_dyna_int_cont}
	\begin{cases}
		\mbox{Find} \ {\bu}_{n} \in V \ \mbox{such that}\\
		\displaystyle{\frac{1}{\Delta t} \int_{\Omega} \rho( \dot{\bu}_{n}-\dot{\bu}_{n-1})\cdot{\bv }\ dx
			+ \int_{\Omega} {\bf \Pi}^{\text{algo}}_{n-\frac{1}{2}}:\nabla{\bv}\ dx -
			\int_{\Gamma_2} {\mathbf f_2}_{n-\frac{1}{2}}\cdot{\bv}\ dx -
			\int_{\Omega}
			{\mathbf f_0}_{n-\frac{1}{2}}\cdot{\bv}\ da}\\
		\quad\quad\quad\quad\quad\quad\quad\quad\quad\quad +\displaystyle{\int_{\Gamma_3} [ \lambda_{\nu_{n-\frac{1}{2}}}
			\dot{\delta}_{\nu_{n-\frac{1}{2}}}+\blambda_{\tau_{n-\frac{1}{2}}}
			\cdot\dot{\bdelta}_{\tau_{n-\frac{1}{2}}}]\ da=0}
		\end{cases}
	\end{eqnarray}
	 In order to obtain energy conservation properties, we propose to change the normal constraint $\lambda_\nu{_{n-\frac{1}{2}}}$ by an improved discrete value (\ref{improvement}) which will respect the energy balance of the continuous case (\ref{cons}).\\
	 \textbf{Discrete form of Improved Normal Compliance (INC) conditions}\\
	In order to  ``almost'' conserve the energy and to respect the energy balance of the continuous case, we replace the normal contact distance $(\delta_{\nu}^n)^{\alpha-1}$ by 
	\begin{eqnarray}\label{xx}
		\widetilde{\delta}_{\nu}^n\coloneqq\displaystyle{\frac{\big[\big(\delta_{\nu}^n\big)_{+}\big]^{\alpha}- \big[\big(\delta_{\nu}^{n-1})_{+}\big]^{\alpha}}{\alpha\big(\delta_{\nu}^n-\delta_{ \nu}^{n-1}\big)}}
		\end{eqnarray}
		We then obtain the normal stress value $\lambda_{\nu_{n-\frac{1}{2}}}$ of the improved normal compliance condition in the discrete case:
		\begin{equation}\label{improvement}
			\lambda_{\nu_{n-\frac{1}{2}}}=c_{\nu}\frac{\alpha}{2}\widetilde{\delta}_{\nu}^n.
		\end{equation}
		\textbf{Discrete Energy Evolution Analysis}\\
		This part is devoted to establishing the energy conservation properties induced by the improved normal compliance described in the previous paragraph.
		We use above the notation $E_n$ and $E_{n-1}$ for the energy $E$ of the hyperelastic system of contact with friction evaluated at times $t_n$ and $t_{n-1}$ respectively.
		For example, the discrete energy at time $t_n$ can be written as follows:
		\begin{equation}\label{energy_dis}
			E_{n} = \frac{1}{2} \int_{\Omega} \rho |\dot{\bf u}^{2}_{n}|\ dx + \int_{\Omega} \widetilde{W} ({\bC_{n}})\ dx.
		\end{equation}
		The general evaluation of the discrete energy of the contact problem with friction between times $t_n$ and $t_{n-1}$ is based on the following proposition.
		\begin{proposition}
			\rm{The following discrete energy balance holds between times $t_n$ and $t_{n-1}$}:
			\begin{eqnarray}
				\displaystyle{E_n-E_{n-1}=\Delta t\langle \fb_{n-\frac{1}{2}},u_{n-\frac{1}{2}}\rangle_{V^* \times V} -\Delta t\int_{\Gamma_3} \big[ \lambda_{ \nu_{n-\frac{1}{2}}}\dot{\delta}_{\nu_{n- \frac{1}{2}}}+
					\blambda_{ \tau_{n-\frac{1}{2}}}\cdot
					\dot{\bdelta}_{ \tau_{n-\frac{1}{2}}}\big]}\ da
			\end{eqnarray}
		\end{proposition}
		\begin{proof}
			Using the variational formulation (\ref{f_dyna_int_cont}) with
			$$\displaystyle{ \bv=\dot{\bu}_{n-\frac{1}{2}}=\frac{\bu_{n}-{\bu}_{n-1}}{\Delta t}=\frac{\dot{\bu}_{n}+{\dot\bu_{n-1}}}{2}}$$ we get the following equality
			\begin{equation}
				\frac{1}{2\Delta t} \int_{\Omega} \rho( \dot{\bu }_{n}-\dot{\bu }_{n-1})\cdot(\dot{\bu }_{n}+\dot{\bu }_{n-1})\ {\rm{dx}}
					+ \frac{1}{\Delta t}\int_{\Omega} {\bf \Pi}^{\text{algo}}:\nabla({\bu_{n}-\bu_{n-1}})\ \rm {dx} \nonumber
				\end{equation}
				\begin{equation}
					\quad\quad\quad=\langle \fb_{n-\frac{1}{2}},\bu_{n-\frac{1}{2}}\rangle_{V^*\times V} +\int_{\Gamma_3} \big[ \lambda_{ \nu_{n-\frac{1}{2}}} \dot{\delta}_{ \nu_{n-\frac{1}{2}}} +\blambda_{ \tau_n}\cdot\dot{\bdelta}_{\tau_{n-\frac{1}{2}}}\big] \ da
				\end{equation}
				Also, using the identity
				$(\dot{\bu}_{n}-\dot{\bu }_{n-1})\cdot(\dot{\bu }_{n}+\dot{\bu }_{n-1 })=[\dot{\bu }_{n}]^2-[\dot{\bu }_{n-1}]^2$
				and the conservation property of the Gonzalez scheme given in equality (\ref{algo}), we get that
\begin{eqnarray}
	\frac{1}{2 \Delta t} \int_{\Omega} \rho\big( [\dot{\bu }_{n}]^2-[\dot{\bu}_{n-1}]^2\big)\ dx
		+ \frac{1}{\Delta t}\int_{\Omega} \big(\widetilde{W}({\bC}_{n})- \widetilde{W}({\bC}_{n-1})\big) \ dx \nonumber \\
		=\langle \fb_{n-\frac{1}{2}},\bu_{n-\frac{1}{2}}\rangle_{V^* \times V} -\int_{\Gamma_3} \big[ \lambda_{ \nu_{n-\frac{1}{2}}}
		\dot{\delta}_{ \nu_{n-\frac{1}{2}}}+ 
		\blambda_{ \tau_{n-\frac{1}{2}}}\cdot
		\dot{\bdelta}_{ \tau_{n-\frac{1}{2}}} \big] \ da
\end{eqnarray}
By using the definition (\ref{energy_dis}) of the discrete energy at times $t_{n-1}$ and $t_n$, we obtain the assertion.
\end{proof}
Using the previous proposition, we can give an estimate of the discrete energy balance for the contact law with Improved Normal Compliance (\ref{improvement}).
When the external forces are assumed to be zero, by using  $\frac{\delta_{\nu}^{n}-{\delta}_{\nu}^{n-1}}{\Delta t}=\dot{\delta}_{{\nu}_{n-\frac{1}{2}}}$ and considering the formula (\ref{xx}), the energy balance is
	\begin{eqnarray}
	E_n-E_{n-1}=-\int_{\Gamma_3} \frac{c_\nu}{2}\big([(\delta_\nu^n)_{+}]^{\alpha}- [(\delta_\nu^{n-1})_{+}]^{\alpha}\big)\ da-\Delta t\int_{\Gamma_3} \blambda_{ \tau_{n-\frac{ 1}{2}}}\cdot\dot{\bdelta}_{ \tau_{n-\frac{1}{2}}}\ da.
\end{eqnarray}
 We notice that the condition of Improved Normal Compliance allows in the case without friction an evaluation of the discrete energy which is in agreement with the continuous case (\ref{cons}).\\
\textbf{Case without friction}\\
The difference $[(\delta_\nu^n)_{+}]^{\alpha}-[(\delta_\nu^{n-1})_{+}]^{\alpha}$ is very small since the penetrations $(\delta_\nu^n)_{+}$ and $(\delta_\nu^{n-1})_{+}$ are also small. So the energy of the system is ``almost'' conserved, i.e. $E_n\approx E_{n-1}$.\\
\textbf{Case with friction}\\
In this case the product  $\blambda_{ \tau_{n-\frac{1}{2}}}\cdot
\dot{\bdelta}_{ \tau_{n-\frac{1}{2}}}$ is always positive due to the friction law, so
we observe a dissipation of the energy: $E_n\leq E_{n-1}$.  In other words this strategy limits the dissipation of energy between times $t_n $ and $t_{n-1}$. \\
In summary, the INC strategy respects the dissipation in the case of friction and ``almost'' conserves the energy in the case of contact without friction, and this is achieved by limiting the penetration.\\
In the following sections, we will propose both semi-smooth Newton's method as well as the Primal Dual Active Set (PDAS) algorithm in the case of Normal Compliance conditions.		
              
\section{Semi-Smooth Newton approach for solving normal compliance conditions}\label{s4}
 
 The semi-smooth Newton/PDAS methods appear to be one of the most relevant methods for solving frictional contact problems (cf \cite{hint2,wohlcoulomb, wohl}). These methods are based on the following principle: the conditions of contact and friction are reformulated in terms of  non-linear complementarity equations whose solution is provided by the semi-smooth iterative method of Newton \cite{hint2, hint}. To this end, we need the generalized derivative of complementary functions for contact and friction. In practice, the conditions of contact with  Coulomb's friction can be formulated in terms of a fixed point problem related to a quasi-optimization one.  From a purely algorithmic point of view, the main goal of these methods is to separate the nodes potentially in contact into two subsets (active and inactive) and to find the correct subset of all the nodes actually in active contact (subset $ \cal A $), as opposed to those that are inactive (subset $ \cal I $). In practice, the semi-smooth Newton/PDAS methods do not require the use of Lagrange multipliers. In fact, the boundary conditions on the subsets $ \cal A $ and $ \cal I $ are directly enforced thanks to the fixed points found by the semi-smooth Newton method, and consequently, their implementation can be achieved without much effort. 

    \subsection{Standard Normal Compliance (SNC) conditions}\label{semi_snc}
    \subsubsection{Complementary function}\label{comp_sn}
The standard normal compliance contact conditions (\ref{compl}) with $\alpha=2$ can be formulated from the following non-linear complementary function: 
\begin{align}
{\cal C}_{\nu}^{\lambda}(\delta_\nu,\lambda_\nu)=\lambda_\nu - c_\nu [\delta_\nu]_+,\label{complementary_sn}
\end{align}
where we have dropped, for the sake of readability, the time index $n+1$.
     \subsubsection{Generalized derivative of complementary functions}\label{gen_deriv_sn}
     \noindent We provide the generalized derivative of the complementary functions in the gap and contact cases.

 $\bullet$ {Gap case: $\delta_\nu\le 0$}.
 
\noindent According to the complementary function ${\cal C}_{\nu}^{\lambda}(\delta_\nu,\lambda_\nu)=\lambda_\nu$  we have 
\begin{align}
&d_{\delta_\nu} {\cal C}_{\nu}^{\lambda}=0\label{phi11_sn},\\
&d_{\lambda_\nu} {\cal C}_{\nu}^{\lambda}=d{\lambda_\nu}.\label{phi12_sn}
\end{align}

$\bullet$ {Contact case:  $\delta_\nu > 0$}.

\noindent Given the complementary function ${\cal C}_{\nu}^{\lambda}(\delta_\nu,\lambda_\nu)=\lambda_\nu - c_\nu \delta_\nu$, we have
\begin{align}
&d_{\delta_\nu} {\cal C}_{\nu}^{\lambda}=-c_\nu d{\delta_\nu}\label{phi11st_sn},\\
&d_{\lambda_\nu} {\cal C}_{\nu}^{\lambda}=d{\lambda_\nu}\label{phi12st_sn}.
\end{align}

\noindent By combining (\ref{phi11_sn})--(\ref{phi12st_sn}), with the generalized derivative ${\cal D}_{{\cal C}_{\nu}^{\lambda}}$ of ${\cal C}_{\nu}^{\lambda}$, we obtain
\begin{align}
&{\cal D}_{{\cal C}_{\nu}^{\lambda}}(\delta_\nu,\lambda_\nu)(d\delta_\nu,d \lambda_\nu)= - c_\nu({\mathbf 1}_{\text{Contact}})d \delta_\nu + d \lambda_\nu,
\end{align}
where 
\begin{align*}
& {\mathbf 1}_{\text{Contact}} = 0 \ {\rm if}\  \delta_\nu\le 0,\\
& {\mathbf 1}_{\text{Contact}} = 1 \ {\rm if}\ \delta_\nu > 0.
\end{align*}

     \subsubsection{Fixed point conditions from Newton's Semi-Smooth approach}\label{fix_pt_sn}
     
     Using now the semi-smooth Newton formalism (indexed by the superscript $k$) at the current fixed point iterate $(\delta^{k}_\nu,\lambda^{k}_\nu)$ of the complementary functions ${\cal C}_{\nu}^\lambda$, one can derive the new iterate $(\delta^{k+1}_\nu,\lambda^{k+1}_\nu)$ as follows:
\begin{align}
&{\cal D}_{{\cal C}_{\nu}^{\lambda}}(\delta^{k}_\nu,\lambda^{k}_\nu)(\Delta \delta^{k+1}_\nu,\Delta \lambda^{k+1}_\nu)= - {\cal C}_{\nu}^{\lambda} (\delta^{k}_\nu,\lambda^{k}_\nu),\label{G_R_np_sn}\\[2mm]
& (\delta^{k+1}_\nu,\lambda^{k+1}_\nu) =(\delta^{k}_\nu,\lambda^{k}_\nu) +(\Delta \delta^{k+1}_\nu,\Delta \lambda^{k+1}_\nu)\nonumber.
\end{align}

 $\bullet$ {Gap case: $ {\mathbf 1}_{\text{Contact}} = 0$}.

\noindent From the equations (\ref{G_R_np_sn}) we have
\begin{align}
&\lambda^{k+1}_\nu-\lambda^{k}_\nu=-\lambda^{k}_\nu.
\end{align}
Next, the gap conditions of the semi-smooth Newton formalism are as follows
\begin{align}
&\lambda^{k+1}_\nu=0.
\end{align}

 $\bullet$ {Contact case: $ {\mathbf 1}_{\text{Contact}} = 1$}.

\noindent From the equations (\ref{G_R_np_sn}) we have
\begin{align}
&- c_\nu(\delta^{k+1}_\nu-\delta^{k}_\nu) + (\lambda^{k+1}_\nu-\lambda^{k}_\nu) = -\lambda^{k}_\nu + c_\nu \delta^{k}_\nu.
\end{align}
Next, 
\begin{align}
& \lambda^{k+1}_\nu = c_\nu \delta^{k+1}_\nu.
\end{align}

   \subsection{Improved Normal Compliance (INC) conditions}\label{semi_inc}
\subsubsection{Complementary function}\label{comp_in}

The improved normal compliance contact condition (\ref{improvement}) can be formulated from the following non-linear complementary function:
\begin{align}
	{\cal C}_{\nu}^{\lambda}(\delta_\nu^{n+1},\lambda_\nu^{n+1})=\lambda_\nu^{n+1} - [ c_\nu \frac{\alpha}{2} \widetilde \delta^{n+1}_\nu]_+.\label{complementary_in}
\end{align}
  For the sake of readability, from now on, we use this equation: $	{\cal C}_{\nu}^{\lambda}(\delta_\nu,\lambda_\nu)=\lambda_\nu - [ c_\nu \frac{\alpha}{2} \widetilde \delta^{n+1}_\nu]_+$.

\subsubsection{Generalized derivative of complementary functions}\label{gen_deriv_in}
\noindent Now, we provide the generalized derivative of the complementary functions in the gap and contact cases.

$\bullet$ {Gap case: $\widetilde \delta^{n+1}_\nu\le 0$}.

\noindent According to the complementary function ${\cal C}_{\nu}^{\lambda}(\delta_\nu,\lambda_\nu)=\lambda_\nu$  we have the following derivative
\begin{align}
	&d_{\delta_\nu} {\cal C}_{\nu}^{\lambda}=0\label{phi11_in},\\
	&d_{\lambda_\nu} {\cal C}_{\nu}^{\lambda}=d{\lambda_\nu}. \label{phi12_in}
\end{align}

$\bullet$ {Contact case:  $\widetilde \delta^{n+1}_\nu > 0$}.

\noindent Given the complementary functions ${\cal C}_{\nu}^{\lambda}(\delta_\nu,\lambda_\nu)=\lambda_\nu - c_\nu \widetilde \delta_\nu$, we have
\begin{align}
	&d_{\delta_\nu} {\cal C}_{\nu}^{\lambda}=-c_\nu \frac{\alpha (([\delta^{n+1}_\nu]_+)^{\alpha-1} - \widetilde \delta^{n+1}_\nu)}{2(\delta^{n+1}_\nu - \delta^{n}_\nu)} d{\delta_\nu}\label{phi11st_in},\\
	&d_{\lambda_\nu} {\cal C}_{\nu}^{\lambda}=d{\lambda_\nu}\label{phi12st_in}.
\end{align}

\noindent By combining (\ref{phi11_in})--(\ref{phi12st_in}), with ${\cal D}_{{\cal C}_{\nu}^{\lambda}}$  the generalized derivative of ${\cal C}_{\nu}^{\lambda}$, we obtain
\begin{align}
	&{\cal D}_{{\cal C}_{\nu}^{\lambda}}(\delta_\nu,\lambda_\nu)(d \delta_\nu,d \lambda_\nu)= -c_\nu \frac{\alpha (([\delta^{n+1}_\nu]_+)^{\alpha-1} - \widetilde \delta^{n+1}_\nu)}{2(\delta^{n+1}_\nu - \delta^{n}_\nu)} ({\mathbf 1}_{\text{Contact}})d \delta_\nu +  d \lambda_\nu,
\end{align}
where 
\begin{align*}
	& {\mathbf 1}_{\text{Contact}} = 0 \ {\rm if}\  \widetilde \delta^{n+1}_\nu \le 0,\\
	& {\mathbf 1}_{\text{Contact}} = 1 \ {\rm if}\ \widetilde \delta^{n+1}_\nu > 0.
\end{align*}

\subsubsection{Fixed point conditions from Newton's Semi-Smooth approach}\label{fix_pt_in}

Using now the semi-smooth Newton formalism (indexed by the superscript $k$) at the current fixed point iterate $(\delta^{k}_\nu,\lambda^{k}_\nu)$ of the complementary functions ${\cal C}_{\nu}^\lambda$, one can derive the new iterate $(\delta^{k+1}_\nu,\lambda^{k+1}_\nu)$
\begin{align}
	&{\cal D}_{{\cal C}_{\nu}^{\lambda}}(\delta^{k}_\nu,\lambda^{k}_\nu)(\Delta \delta^{k+1}_\nu,\Delta \lambda^{k+1}_\nu)= - {\cal C}_{\nu}^{\lambda} (\delta^{k}_\nu,\lambda^{k}_\nu),\label{G_R_np_in}\\[2mm]
	& (\delta^{k+1}_\nu,\lambda^{k+1}_\nu) =(\delta^{k}_\nu,\lambda^{k}_\nu) +(\Delta \delta^{k+1}_\nu,\Delta \lambda^{k+1}_\nu)\nonumber.
\end{align}

$\bullet$ {Gap case: $ {\mathbf 1}_{\text{Contact}} = 0$}

\noindent From the equations (\ref{G_R_np_in}) we have
\begin{align}
	&\lambda^{k+1}_\nu-\lambda^{k}_\nu=-\lambda^{k}_\nu.
\end{align}
Next, the gap conditions of the semi-smooth Newton formalism are as follows
\begin{align}
	&\lambda^{k+1}_\nu=0.
\end{align}

$\bullet$ {Contact case: $ {\mathbf 1}_{\text{Contact}} = 1$}

\noindent From the equations (\ref{G_R_np_in}) we have
\begin{align}
	& -c_\nu \frac{\alpha (([\delta^{k,n+1}_\nu]_+)^{\alpha-1} - \widetilde \delta^{k,n+1}_\nu)}{2(\delta^{k,n+1}_\nu - \delta^{n}_\nu)} (\delta^{k+1,n+1}_\nu-\delta^{k,n+1}_\nu) + (\lambda^{k+1}_\nu-\lambda^{k}_\nu) = -\lambda^{k}_\nu +  c_\nu \frac{\alpha}{2} \widetilde \delta^{k,n+1}_\nu.
\end{align}
Next, 
\begin{align}
	& \lambda^{k+1}_\nu =  c_\nu \frac{\alpha}{2} \widetilde \delta^{k,n+1}_\nu + c_\nu \frac{\alpha (([\delta^{k,n+1}_\nu]_+)^{\alpha-1} - \widetilde \delta^{k,n+1}_\nu)}{2(\delta^{k,n+1}_\nu - \delta^{n}_\nu)} (\delta^{k+1,n+1}_\nu-\delta^{k,n+1}_\nu).
\end{align}

   \subsection{Compliance for friction conditions}\label{semi_fcc}
\subsubsection{Complementary function}\label{comp_fcf}

The  compliance for friction conditions (\ref{compl1}) can be formulated from the following non-linear complementary function ${\cal C}_{\tau}^{\lambda}(\delta_\nu^{n+1},\dot\bdelta_\tau^{n+1},\lambda_\nu^{n+1},\blambda_\tau^{n+1})$
\begin{align}\label{ccoulomb}
	&\hspace{-2mm}{\cal C}_{\tau}^{\lambda}(\delta_\nu^{n+1},\dot\bdelta_\tau^{n+1},\lambda_\nu^{n+1},\blambda_\tau^{n+1})=\max ( \mu \lambda_{\nu}^{n+1}, \| c_\tau \dot\bdelta_{\tau}^{n+1}\|)\blambda_{\tau}^{n+1}- \mu \lambda_{\nu}^{n+1}(c_\tau \dot\bdelta_{\tau}^{n+1}).
\end{align}
For the sake of readability, in the following we use this equation: $${\cal C}_{\tau}^{\blambda}(\delta_{\nu},\dot\bdelta_{\tau},\lambda_{\nu},\blambda_{\tau})=\max ( \mu \lambda_{\nu}, \|c_\tau \dot\bdelta_{\tau}\|)\blambda_{\tau}- \mu \lambda_{\nu}(c_\tau \dot\bdelta_{\tau}).$$

\subsubsection{Generalized derivative of complementary functions}\label{gen_deriv_in}
\noindent Now, we provide the generalized derivative of the complementary function in the gap and friction cases.\\
$\bullet$ {Gap case}: $ \blambda_{\tau} = \mathbf 0$, ${\cal C}_{\tau}^{\blambda}(\delta_{\nu},\dot\bdelta_{\tau},\lambda_{\nu},\blambda_{\tau})= \mathbf 0$. \\[0.2cm]
$\bullet$ {Stick case}: $\|\blambda_{\tau}\| < \mu \lambda_{\nu}$:\\
\begin{align*}
	&{\cal C}_{\tau}^{\blambda}(\delta_{\nu},\dot\bdelta_{\tau},\lambda_{\nu},\blambda_{\tau})=\mu \lambda_{\nu} \blambda_{\tau}  -\mu c_\tau \lambda_{\nu}\dot\bdelta_{\tau}.
\end{align*}
Then
\begin{align}
	&d_{\delta_{\nu}} {\cal C}_{\nu}^{\blambda}=0, \label{phi21st}\\
	&d_{\dot\bdelta_{\tau}} {\cal C}_{\tau}^{\blambda}=-\mu c_\tau \lambda_{\nu}d\dot\bdelta_{\tau},\label{phi22st}\\
	&d_{\lambda_{\nu}} {\cal C}_{\tau}^{\blambda}=(\mu \blambda_{\tau}-\mu c_\tau \dot\bdelta_{\tau})d{\lambda_{\nu}},\label{phi23st}\\
	&d_{\blambda_{\tau}} {\cal C}_{\tau}^{\blambda}= \mu \lambda_{\nu} d{\blambda_{\tau}}\label{phi24st}.
\end{align}
$\bullet$ {Slip case}: $\|\blambda_{\tau}\| \ge \mu \lambda_{\nu}$
\begin{align*}
	&{\cal C}_{\tau}^{\blambda}(\delta_{\nu},\dot\bdelta_{\tau},\lambda_{\nu},\blambda_{\tau})=\|c_\tau \dot\bdelta_{\tau}\|\blambda_{\tau}- \mu\lambda_{\nu} (c_\tau \dot\bdelta_{\tau}).
\end{align*}
Then
\begin{align}
	&d_{\delta_{\nu}} {\cal C}_{\tau}^{\blambda}=0,\label{phi21sl}\\
	&d_{\dot\bdelta_{\tau}} {\cal C}_{\tau}^{\blambda}=\Big(c_\tau \blambda_{\tau}\frac{( \dot\bdelta_{\tau})^T}{\| \dot\bdelta_{\tau}\|}-\mu c_\tau \lambda_{\nu}\bI_2\Big)d{\dot\bdelta_{\tau}},\label{phi22sl}\\
	&d_{\lambda_{\nu}} {\cal C}_{\tau}^{\blambda}=- \mu c_\tau \dot\bdelta_{\tau}d\lambda_{\nu},\label{phi23sl}\\
	&d_{\blambda_{\tau}} {\cal C}_{\tau}^{\blambda}=\|c_\tau \dot\bdelta_{\tau}\|d{\blambda_{\tau}}\label{phi24sl}.
\end{align}

\subsubsection{Fixed point conditions from Newton's Semi-Smooth approach}\label{fix_pt_in}

\noindent By combining (\ref{phi21st})--(\ref{phi24sl}), with ${\cal G}_{{\cal C}_{\tau}^{\blambda}}$ the generalized derivative of  ${\cal C}_{\tau}^{\blambda}$, respectively, we obtain
\begin{align}
	&{\cal G}_{{\cal C}_{\tau}^{\blambda}}(\delta_{\nu},\dot\bdelta_{\tau},\lambda_{\nu},\blambda_{\tau})(\Delta \delta_{\nu},\Delta\dot\bdelta_{\tau},\Delta\lambda_{\nu},\Delta\blambda_{\tau})= 
	 {\mathbf 1}_{\text{Stick}} \Big( \mu \blambda_{\tau} \Big)\Delta{\lambda_{\nu}} \nonumber\\
	&+ {\mathbf 1}_{\text{Slip}} \Big(c_\tau \blambda_{\tau}\frac{( \dot\bdelta_{\tau})^T}{\| \dot\bdelta_{\tau}\|}\Big)\Delta{\dot\bdelta_{\tau}}  - \mu c_\tau \dot\bdelta_{\tau}\Delta\lambda_{\nu} -\mu c_\tau \lambda_{\nu}\Delta{\dot\bdelta_{\tau}} \nonumber\\
	& + {\mathbf 1}_{\text{Stick}}\Big(\mu\lambda_{\nu}\Big)\Delta{\blambda_{\tau}}  + {\mathbf 1}_{\text{Slip}}\Big(\|c_\tau \dot\bdelta_{\tau}\| \Big)\Delta{\blambda_{\tau}} \nonumber
\end{align}
where 
\begin{align*}
	& {\mathbf 1}_{\text{Stick}} = 1, {\mathbf 1}_{\text{Slip}} = 0\ {\rm if}\ \|\blambda_{\tau}\| < \mu \lambda_{\nu},\\
	& {\mathbf 1}_{\text{Stick}} = 0,  {\mathbf 1}_{\text{Slip}} = 1\ {\rm if}\ \|\blambda_{\tau}\| \ge \mu \lambda_{\nu}.
\end{align*}
Using now the semi-smooth Newton formalism at the current iterate $(\delta_{\nu}^{(k)},\dot\bdelta_{\tau}^{(k)},\lambda_{\nu}^{(k)},\blambda_{\tau}^{(k)})$, one can derive the new iterate $(\delta_{\nu}^{(k+1)},\dot\bdelta_{\tau}^{(k+1)},\lambda_{\nu}^{(k+1)},\blambda_{\tau}^{(k+1)})$ as follows:
\begin{align}
	&{\cal G}_{{\cal C}_{\tau}^{\blambda}}(\delta_{\nu}^{(k)},\dot\bdelta_{\tau}^{(k)},\lambda_{\nu}^{(k)},\blambda_{\tau}^{(k)})(\Delta \delta_{\nu}^{(k+1)},\Delta\dot\bdelta_{\tau}^{(k+1)},\Delta\lambda_{\nu}^{(k+1)},\Delta\blambda_{\tau}^{(k+1)})
	= - {\cal C}_{\tau}^{\blambda} (\delta_{\nu}^{(k)},\dot\bdelta_{\tau}^{(k)},\lambda_{\nu}^{(k)},\blambda_{\tau}^{(k)})\nonumber,\\[2mm]
	& (\delta_{\nu}^{(k+1)},\dot\bdelta_{\tau}^{(k+1)},\lambda_{\nu}^{(k+1)},\blambda_{\tau}^{(k+1)})
	 =(\delta_{\nu}^{(k)},\dot\bdelta_{\tau}^{(k)},\lambda_{\nu}^{(k)},\blambda_{\tau}^{(k)}) +(\Delta \delta_{\nu}^{(k+1)},\Delta\dot\bdelta_{\tau}^{(k+1)},\Delta\lambda_{\nu}^{(k+1)},\Delta\blambda_{\tau}^{(k+1)})\nonumber.
\end{align}

\noindent$\bullet$ {Stick case}: $ {\mathbf 1}_{\text{Stick}} = 1, {\mathbf 1}_{\text{Slip}} = 0$.

We have
\begin{align*}
	& -\mu c_\tau \lambda_{\nu}^{(k)}(\dot\bdelta_{\tau}^{(k+1)}-\dot\bdelta_{\tau}^{(k)}) +  \mu\lambda_{\nu}^{(k)} (\blambda_{\tau}^{(k+1)}-\blambda_{\tau}^{(k)}) + (\mu \blambda_{\tau}^{(k)}-\mu c_\tau \dot\bdelta_{\tau}^{(k)}) (\lambda_{\nu}^{(k+1)}-\lambda_{\nu}^{(k)}) \\
	& = - \mu\lambda_{\nu}^{(k)} (\blambda_{\tau}^{(k)} +\mu c_\tau \lambda_{\nu}^{(k)}\dot\bdelta_{\tau}^{(k)}.
\end{align*}
Next, with  $\blambda_{\tau}^{(k)} = c_\tau \dot\bdelta_{\tau}^{(k)}$, we obtain
\begin{align}
	& \blambda_{\tau}^{(k+1)} =c_\tau \dot\bdelta_{\tau}^{(k+1)}.
\end{align}

\noindent$\bullet$ {Slip case}: $ {\mathbf 1}_{\text{Stick}} = 0, {\mathbf 1}_{\text{Slip}} = 1$.

We obtain
\begin{align}
	&\Big(c_\tau \blambda_{\tau}^{(k)}\frac{( \dot\bdelta_{\tau}^{(k)})^T}{\| \dot\bdelta_{\tau}^{(k)}\|}-\mu c_\tau \lambda_{\nu}^{(k)}\bI_2\Big)(\dot\bdelta_{\tau}^{(k+1)}-\dot\bdelta_{\tau}^{(k)}) \label{phi2c}\\
	&  - \mu (c_\tau \dot\bdelta_{\tau}^{(k)})(\lambda_{\nu}^{(k+1)}-\lambda_{\nu}^{(k)}) +\Big( \|c_\tau \dot\bdelta_{\tau}^{(k)}\|  \Big)(\blambda_{\tau}^{(k+1)}-\blambda_{\tau}^{(k)})\nonumber\\
	&= -\|c_\tau \dot\bdelta_{\tau}^{(k)}\|\blambda_{\tau}^{(k)}+ \mu\lambda_{\nu}^{(k)} (c_\tau \dot\bdelta_{\tau}^{(k)})\nonumber.
\end{align}
Therefore, after an elementary computation, with  $\blambda_{\tau}^{(k)} =\mu \lambda_{\nu}^{(k)} \frac{ \dot\bdelta_{\tau}^{(k)}}{\| \dot\bdelta_{\tau}^{(k)}\|}$ we have 
\begin{align*}
	&\blambda_{\tau}^{(k+1)} =  \mu\lambda_{\nu}^{(k+1)} \frac{ \dot\bdelta_{\tau}^{(k)}}{\| \dot\bdelta_{\tau}^{(k)}\|} - \Big( \blambda_{\tau}^{(k)}\frac{( \dot\bdelta_{\tau}^{(k)})^T}{\| \dot\bdelta_{\tau}^{(k)}\|^2}-\mu \lambda_{\nu}^{(k)} \frac{\bI_2}{\| \dot\bdelta_{\tau}^{(k)}\|} \Big) (\dot\bdelta_{\tau}^{(k+1)}-\dot\bdelta_{\tau}^{(k)})\nonumber.
\end{align*}
For a two dimensional case, one obtains a simplified version of the previous condition: 
\begin{align*}
	&\blambda_{\tau}^{(k+1)} =  \mu\lambda_{\nu}^{(k+1)} \frac{ \dot\bdelta_{\tau}^{(k)}}{\| \dot\bdelta_{\tau}^{(k)}\|}\nonumber.
\end{align*}
For more details regarding the obtention of such a simplified version, we refer to the proof provided in \cite{ABCD2}.
     
\section{Primal-Dual Active Set methods}\label{s5}
    This section is devoted to the numerical treatment of the contact conditions using a Primal-Dual Active Set method within the framework of dynamic contact problems. After defining the active and inactive subsets of all nodes that are potentially in contact, we compute the contact conditions on each subset only in terms of contact reaction, using the local general equations of motion.

\subsection{Primal-Dual Active Set method for standard normal compliance}\label{activ_snc}
Let us denote by ${\cal S}$ the set of potential contact  and $\gamma$ a potential contact node  belonging to ${\cal S}$. The standard normal contact condition  (\ref{compl}) with $\alpha=2$ is enforced by applying an active set strategy which derives directly from the computation of the fixed point on the non-linear complementary functions ${\cal C}_{\nu}^{\blambda}$ and ${\cal C}_{\tau}^{\blambda}$  based on the Newton semi-smooth scheme. The active and inactive sets are defined as follows

\begin{align*}
&{\cal A}_{\nu}^{k+1}=\{\gamma\in {\cal S}: \delta^{\gamma,k}_\nu \geq 0\},\\
&{\cal I}_{\nu}^{k+1}=\{\gamma\in {\cal S}:\delta^{\gamma,k}_\nu < 0\},\\
&{\cal A}_{\tau}^{k+1}=\{p\in {\cal S}:\|\blambda_{\tau}^{\gamma,k}\| < \mu \lambda_{\nu}^{\gamma,k}\},\\
&{\cal I}_{\tau}^{k+1}=\{p\in {\cal S}:\|\blambda_{\tau}^{\gamma,k}\| \ge \mu \lambda_{\nu}^{\gamma,k}\}.
\end{align*}
\noindent The status of a given potential $\gamma$ at the non-linear iteration $k$ depends on the set it belongs to. It can be either in the non-contact or frictional contact status (either stick or slip status). It yields the following Algorithm \ref{PDAS_SNC}\\

\begin{algorithm}
	\caption{PDAS for standard normal compliance}\label{PDAS_SNC}
\qquad(i) Choose $(\bdelta^{(0)},\blambda^{(0)})$, $c_{\nu}>0$, $c_{\tau}>0$  and set $k=0$.

\qquad(ii) Compute: $\tau^{\gamma}_\nu =  \delta^{\gamma,k}_\nu$ and  $\tau^{\gamma}_\tau =  - \|\blambda_{\tau}^{\gamma,k}\| + \mu \lambda_{\nu}^{\gamma,k}$ \ \ for each $\gamma \in {\cal S}$.

\qquad(iii) Set the active and inactive sets:
\begin{align*}
&{\cal A}_{\nu}^{k+1}=\{\gamma\in {\cal S}:\tau^{\gamma}_\nu \geq 0\},\\
&{\cal I}_{\nu}^{k+1}={\cal S}\setminus {\cal A}_{\nu}^{k+1}, \\
&{\cal A}_{\tau}^{k+1}=\{p\in {\cal S}:\tau^{\gamma}_\tau > 0 \},\\
&{\cal I}_{\tau}^{k+1}={\cal S}\setminus {\cal A}_{\tau}^{k+1}.
\end{align*}

\qquad(iv) Find $(\bdelta^{\gamma,k+1},\blambda^{\gamma,k+1 })$ such that
\begin{eqnarray}
	&\label{Inu_exact_1}\lambda^{\gamma,k+1}_{\nu}=0, \qquad \blambda^{k+1}_{\tau,p}=\bzero \quad \forall \gamma \in {\cal I}_{\nu}^{k+1},\\[1mm]
&\label{Anu_exact_1}\lambda^{\gamma,k+1}_{\nu}=c_\nu  \delta^{\gamma,k+1}_\nu \qquad \forall \gamma \in {\cal A}_{\nu}^{k+1},\\[1mm]
&\label{Atau_1} \blambda_{\tau}^{\gamma,k+1} = c_\tau \dot\bdelta_{\tau}^{\gamma,k+1} \quad\forall  \gamma \in {\cal A}_{\tau}^{k+1}\cap {\cal A}_{\nu}^{k+1},\\[1mm]
&\label{Itau_1} \blambda_{\tau}^{\gamma,k+1} =  \mu\lambda_{\nu}^{\gamma,k+1} \frac{ \dot\bdelta_{\tau}^{\gamma,k}}{\| \dot\bdelta_{\tau}^{\gamma,k}\|} - \Big( \blambda_{\tau}^{\gamma,k}\frac{( \dot\bdelta_{\tau}^{\gamma,k})^T}{\| \dot\bdelta_{\tau}^{\gamma,k}\|^2}-\mu \lambda_{\nu}^{\gamma,k} \frac{\bI_2}{\| \dot\bdelta_{\tau}^{\gamma,k}\|} \Big) (\dot\bdelta_{\tau}^{\gamma,k+1}-\dot\bdelta_{\tau}^{\gamma,k}) \ \forall \gamma \in {\cal I}_{\tau}^{k+1}\cap {\cal A}_{\nu}^{k+1}. \label{robin}
\end{eqnarray}

\qquad(v) If $\|(\bdelta^{\gamma,k+1},\blambda^{\gamma,k+1})-(\bdelta^{\gamma, k},\blambda^{\gamma,k})\|\leq\epsilon$, ${\cal A}_{\nu}^{k+1}={\cal A}_{\nu}^{k}$ and ${\cal A}_{\tau}^{k+1}={\cal A}_{\tau}^{k}$  stop, else go to (ii).
\end{algorithm}

\subsection{Primal-Dual Active Set method for improved normal compliance}\label{activ_inc}
Likewise, using similar notation, the improved contact condition (\ref{improvement}) are enforced by applying another active set strategy yielding a different Algorithm \ref{PDAS_INC}, described below.

\begin{algorithm}
	\caption{PDAS for improved normal compliance}\label{PDAS_INC}
\qquad(i) Choose $(\bdelta^{(0)},\blambda^{(0)})$, $c_{\nu}>0$, $c_{\tau}>0$  and set $k=0$.

\qquad(ii) Compute: $\tau^{\gamma}_\nu =  \widetilde \delta^{\gamma,k,n+1}_\nu$ and  $\tau^{\gamma}_\tau =  - \|\blambda_{\tau}^{\gamma,k}\| + \mu \lambda_{\nu}^{\gamma,k}$ for each $\gamma \in {\cal S}$.

\qquad(iii) Set the active and inactive sets:
\begin{align*}
	&{\cal A}_{\nu}^{k+1}=\{\gamma\in {\cal S}:\tau^{\gamma}_\nu \geq 0\},\\
	&{\cal I}_{\nu}^{k+1}={\cal S}\setminus {\cal A}_{\nu}^{k+1}, \\
	&{\cal A}_{\tau}^{k+1}=\{p\in {\cal S}:\tau^{\gamma}_\tau > 0 \},\\
	&{\cal I}_{\tau}^{k+1}={\cal S}\setminus {\cal A}_{\tau}^{k+1}.
\end{align*}

\qquad(iv) Find $(\bdelta^{\gamma,k+1},\blambda^{\gamma,k+1 })$ such that
\begin{eqnarray}
	&\label{Inu_exact_2}\lambda^{\gamma,k+1}_{\nu}=0, \qquad \blambda^{\gamma,k+1}_{\tau,p}=\bzero \qquad\forall \gamma\in {\cal I}_{n}^{k+1},\\[1mm]
	&\label{Anu_exact_2}\lambda^{\gamma,k+1}_{\nu}= c_\nu \frac{\alpha}{2} \widetilde \delta^{\gamma,k,n+1}_\nu + c_\nu \frac{\alpha (([\delta^{\gamma,k,n+1}_\nu]_+)^{\alpha-1} - \widetilde \delta^{\gamma,k,n+1}_\nu)}{2(\delta^{\gamma,k,n+1}_\nu - \delta^{\gamma,n}_\nu)} (\delta^{\gamma,k+1,n+1}_\nu-\delta^{\gamma,k,n+1}_\nu) \ \ \forall \gamma \in {\cal A}_{\nu}^{k+1}, \\[1mm]
	&\label{Atau_2} \blambda_{\tau}^{\gamma,k+1} = c_\tau \dot\bdelta_{\tau}^{\gamma,k+1} \quad \forall\gamma \in {\cal A}_{\tau}^{k+1}\cap {\cal A}_{\nu}^{k+1},\\[1mm]
	&\label{Itau_2} \blambda_{\tau}^{\gamma,k+1} =  \mu\lambda_{\nu}^{\gamma,k+1} \frac{ \dot\bdelta_{\tau}^{\gamma,k}}{\| \dot\bdelta_{\tau}^{\gamma,k}\|} - \Big( \blambda_{\tau}^{\gamma,k}\frac{( \dot\bdelta_{\tau}^{\gamma,k})^T}{\| \dot\bdelta_{\tau}^{\gamma,k}\|^2}-\mu \lambda_{\nu}^{\gamma,k} \frac{\bI_2}{\| \dot\bdelta_{\tau}^{\gamma,k}\|} \Big) (\dot\bdelta_{\tau}^{\gamma,k+1}\!\!-\dot\bdelta_{\tau}^{\gamma,k}) \ \ \forall \gamma \in {\cal I}_{\tau}^{k+1}\cap {\cal A}_{\nu}^{k+1}. \label{robin}
\end{eqnarray}

\qquad(v) If $\|(\bdelta^{\gamma,k+1},\blambda^{\gamma,k+1})-(\bdelta^{\gamma, k},\blambda^{\gamma,k})\|\leq\epsilon$, ${\cal A}_{\nu}^{k+1}={\cal A}_{\nu}^{k}$ and ${\cal A}_{\tau}^{k+1}={\cal A}_{\tau}^{k}$ stop, else goto (ii).
\end{algorithm}

Note that the main difference between Algorithms \ref{PDAS_SNC} and \ref{PDAS_INC} lies in step (iv), regarding the normal compliance handling (respectively for \ref{Anu_exact_1} and \ref{Anu_exact_2}), as one would expect.  For the convergence of this type of Semi-Smooth Newton algorithm, we can refer to the following papers: [2], [20], [23].

\section{Numerical experiments}\label{s6}

In what follows, we carry out a comparative study of methods of the energy conservation type, focusing more particularly on the behavior
of the discrete energy of the system during and after the impact.
To do this, we consider two numerical examples, the first concerns
the impact of an elastic ball and the second represents the impact of a
hyperelastic ring. The goal is to demonstrate that the Active Set--Improved Normal Compliance method respects the conservation of energy after impact and practically guarantees non-penetration.
\subsection{Impact of a linear elastic ball against a foundation}
This representative benchmark problem describes the frictionless impact of a linear elastic ball against a foundation. The elastic ball is launched with an initial velocity ($\bu_1=(0,-10)\,m/s$) towards the foundation
$\left\{(x_1,x_2)\in \mathbb{R}^2:\ x_2 \leq 0 \right\}$.
The domain $\Omega$ represents the cross-section of the ball, under the assumption of plane stress. 
The behavior of the material is described by a linear elastic constitutive law defined by the energy density
\begin{equation}
	W^e(\bvarepsilon)=\frac{\lambda^*}{2}({\mathrm{tr}}\,
	\bvarepsilon)^2 +\mu \,{\mathrm{tr}}(\bvarepsilon^2), \quad \forall
	\bvarepsilon \in \mathbb{M}^d.
\end{equation}\label{elastic_density}
with $$\displaystyle\lambda^*=\frac{2\lambda\mu}{\lambda+2\mu},\quad \mu=\frac{E\kappa}{2(1+\kappa)} \quad and \quad \lambda=\frac{E\kappa}{2(1+\kappa)(1-2\kappa)}$$
Here, $E$ and $\kappa$ are respectively the Young modulus and the Poisson ratio of the material and $\mathrm{tr}(\cdot)$ is the trace operator. Note that $\bvarepsilon = \frac{1}{2} ({\nabla \bu}^T + \nabla \bu) $ represents the linearized strain tensor within the framework of the theory of small strains ($\|\bu\| \ll 1$ and
$\|\nabla\bu\| \ll 1$ in $\Omega$).
The physical setting is shown in Figure \ref{fig_num1}. Here:
\begin{align*}
	& \Omega\, =\left\{(x_1,x_2)\in \mathbb{R}^2: \ (x_1-100)^2
	+ (x_2-100)^2 \leq 100 \right\},\\
	&\Gamma_1 = \varnothing,\quad \Gamma_2= \varnothing,\\
	&\Gamma_3=\left\{(x_1,x_2)\in \mathbb{R}^2: \ (x_1-100)^2
	+ (x_2-100)^2 = 100 \right\}.
\end{align*}
\begin{figure}[!h]
	\begin{center}
		\includegraphics[width=5cm]{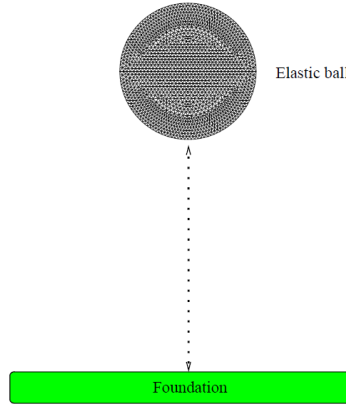}
	\end{center}
	\caption{Discretization of the elastic ball in contact with a foundation.}
	\label{fig_num1}
\end{figure}
We assume that the volume forces do not act on the body during the process. For the discretization of the problem of contact represented in Figure \ref{fig_num1}, we use 7820 elastic nodes. For numerical experiments, we use the following data:
\begin{eqnarray}
	\begin{array}{l}
		\rho = 1000\, kg/m^3, \quad T=2\,s, \quad \Delta t=0.001\,s, \nonumber \\[2mm]
		\bu_0=(0,0)\,m, \quad \bu_1=(0,-10)\,m/s, \nonumber \\[2mm]
		E= 100\,GPa,\quad \kappa=0.35, \quad {\fb}_0=(0,0)\,Pa, \nonumber \\[2mm]
		g=50\,m,  \quad \mu = 0. \nonumber
	\end{array}
\end{eqnarray}
In Figure \ref{defcont}, the successive positions of the deformed ball as well as the contact forces are represented before, during, and after the impact.
\begin{figure}[!h]
	\begin{center}
		\includegraphics[width=11.5cm]{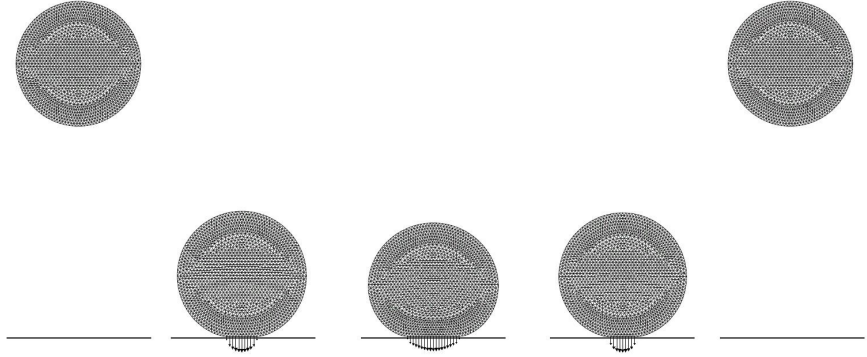}
	\end{center}
	\caption{Sequence of deformed ball and contact forces before, during, and after impact.}
	\label{defcont}
\end{figure}
The interest of this representative example lies in the comparison of the numerical results of the Improved Normal Compliance--PDAS methods  ($\alpha=2$) with other classical numerical methods. For this, we consider five existing methods: \\
\hspace*{0.5cm}- The classical quasi-Lagrangian method with the Signorini contact condition.\\
\hspace*{0.5cm}- The penalty method with a standard normal compliance condition of the form ${\lambda_{{\boldsymbol \nu}}} = c_\nu ({\delta}_{\boldsymbol \nu})_+$. \\
\hspace*{0.5cm}- The Equivalent Mass Matrix (EMM) method proposed by Khenous [32], which represents a specific distribution of the mass matrix without any inertia of the contact nodes.
This method is characterized by relevant stability properties of the contact stress.\\
\hspace*{0.5cm}- Newton's adapted continuity method, developed by Ayyad and Barboteu [2], which is characterized by enforcing, after two steps, the unilateral contact law and the persistence condition during each time increment.\\
\hspace*{0.5cm}- The Active Set method with the persistent contact law.
\begin{figure}[!h]
	\begin{center}
		\includegraphics[width=18cm]{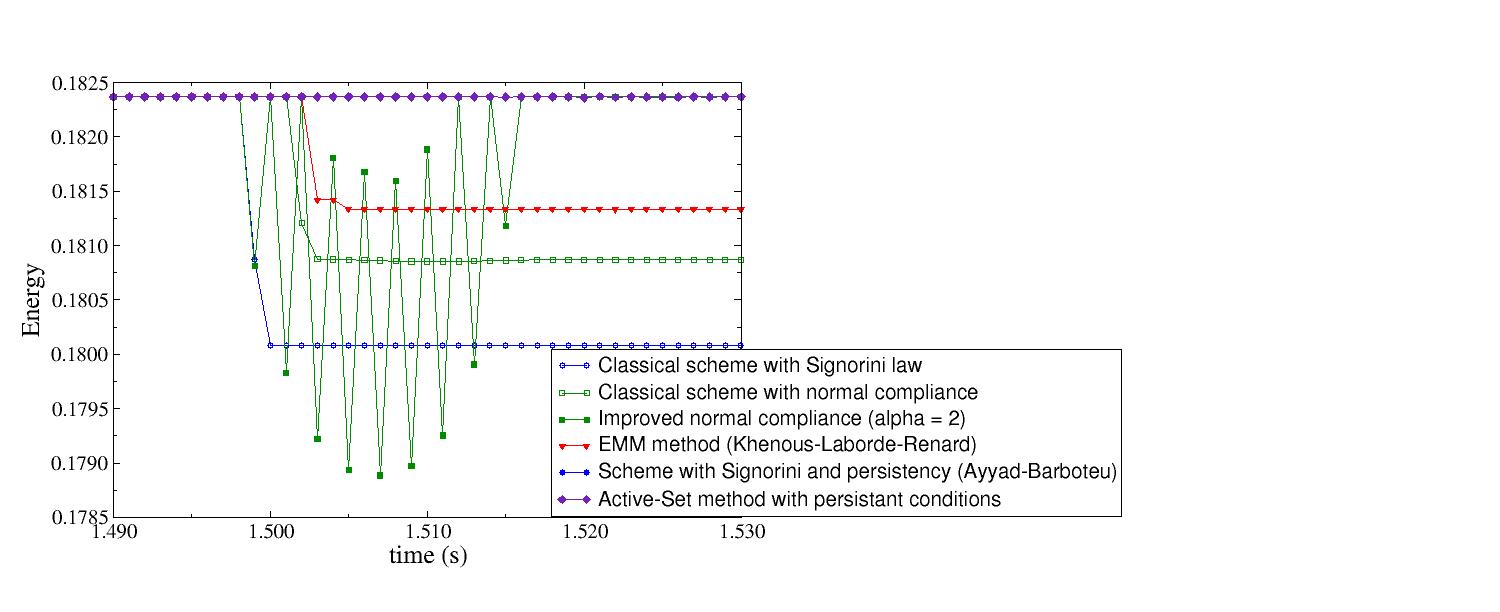}
	\end{center}
	\caption{Discrete energy behavior of selected time integration schemes during impact ($\Delta t=0.001$, $c_\nu=\frac{1}{1000}$).}
\label{ene_ball_1}
\end{figure}

In what follows, we analyze the methods in terms of discrete energy evolution. For this we introduce the total discrete energy at time $t_{n}$ which is given by the following formula:
$$
E_{n} = \frac{1}{2} \int_{\Omega} \rho |\dot{\bu}^{2}_{n}| d\bx +
\int_{\Omega} {\boldsymbol \sigma}^e_{n}:
\bvarepsilon(\bu_{n}) d\bx,
$$
where ${\boldsymbol \sigma}^e= \frac{\partial W^e}{\partial \bvarepsilon}$ denotes the stress tensor for infinitesimal strains.\\
Figure \ref{ene_ball_1} represents the evolution of the total discrete energy of the dynamical system. We note that after the impact (i.e. for $t \ge 1.52\ s$) and for the considered time step $ \Delta t = 1.0\ e^{-3}\ s$, the classical quasi-Lagrangian method with Signorini's law (curve $\circleddash$) as well as the method with standard normal compliance condition (curve $\boxminus$) are characterized by non-conservation of energy, which is not realistic from a physical point of view. We also notice that the EMM method (curve $\blacktriangledown$) strongly reduces the energy dissipation, without obtaining the exact conservation. It turns out the schemes developed by Ayyad and Barboteu [1]
(curve {\Large\textbullet}) and the Improved Normal Compliance--PDAS method (curve~$\blacksquare$) do conserve energy after impact. However, for the penalized method, we find energy fluctuations which disappear after the impact. Besides. for the penalized method and the method used in [1], the unilateral contact is not exactly satisfied, see Table \ref{tab}. Indeed, the penalized method (standard normal compliance) generates a maximum error on the displacement of normal contact of $1.4 e^{-4}\ m$ and $5.1 e^{-3}\ m$ for the method of Ayyad and Barboteu [1]. The Improved Normal Compliance--PDAS method allows to obtain a better energy conservation and in addition to limit the penetration: $5.7 e^{-4}\ m$. The Active Set method for persistent contact (shown by the curve~$\blacklozenge$) enforces exactly the energy conservation without any fluctuation. Due to the ``leapfrog" time step predictor, this Active Set method generates a maximum error on the normal contact displacement of $1.54 e^{-2}\ m$, larger than all the other methods.
\begin{table}[!ht]
\begin{center}
\begin{tabular}{lc}
	\hline
	Methods & Maximum error on the $\delta_\nu\ (m)$ \\
	\hline
	Quasi-Lagrangian with Signorini law & 0. \\
	Scheme with Signorini and persistent (Ayyad--Barboteu) & $5.1 e^{-3}$ \\
	Active set with persistent conditions & $1.54 e^{-2}$ \\
	Active set with classical normal compliance & $1.4 e^{-4}$ \\
	Active set with improved normal compliance ($\alpha = 2$) & $5.7 e^{-4}$ \\
	Active set with improved normal compliance ($\alpha = 3$ )& $8.05 e^{-3}$ \\
	\hline
\end{tabular}
\end{center}
\caption{Maximum error on normal contact displacement ($\Delta t=0.001$, $c_\nu=1 e^{3}$).}
\label{tab}
\end{table}
\begin{figure}[h!]
	\begin{center}
		\includegraphics[width=18cm]{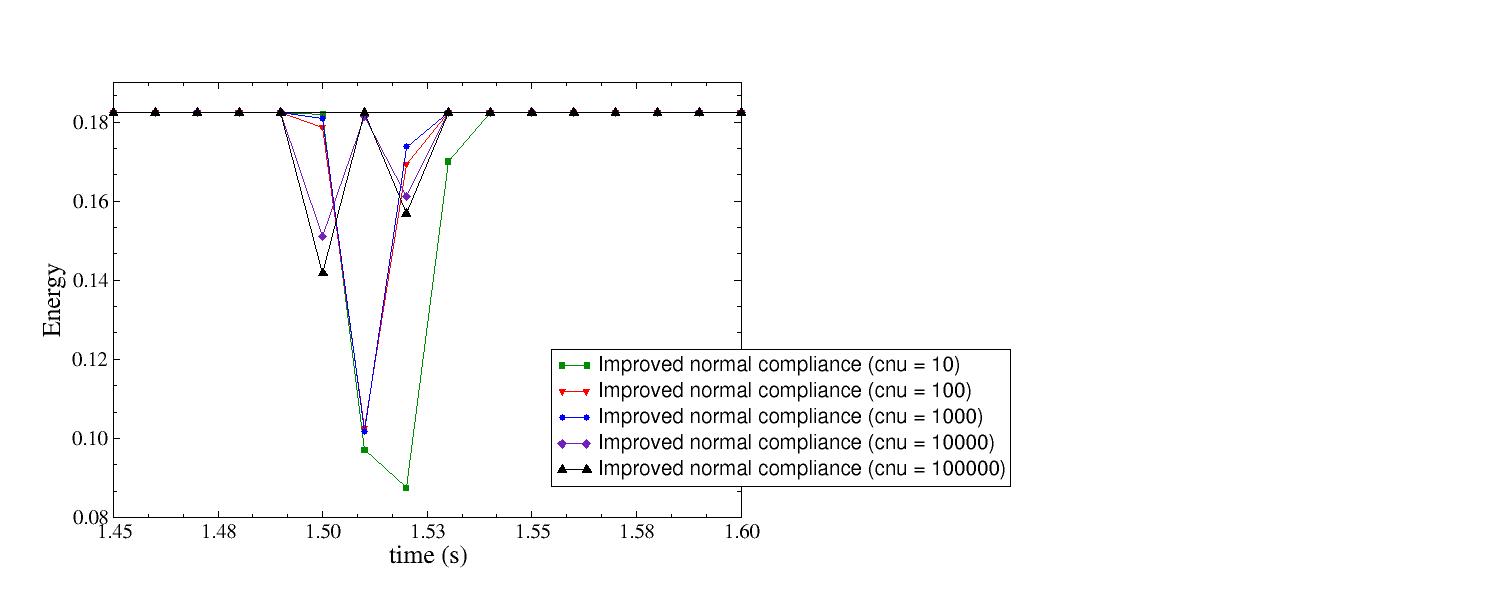}
	\end{center}
	\caption{Discrete energy behavior of the Active Set scheme with Improved Normal Compliance ($\alpha =3$) as a function of $c_\nu$ ($\Delta t=0.01$).}
	\label{ene_ball_2}
\end{figure}
In order to overcome these difficulties, we consider the Active~Set method with improved normal compliance, by analyzing the behavior of discrete energy with respect to several parameters ($c_\nu$, $\alpha$, $\Delta t$). \\
Figure~\ref{ene_ball_2} allows to evaluate the influence of the normal compliance parameter $c_\nu$. We see that, for $c_\nu>1e^{4}$, this method respects the conservation of energy after impact and practically guarantees the non-penetration for $\alpha=2$. Nevertheless, this method generates slight fluctuations of the discrete energy during the impact. For $c_\nu<1e^{4}$, the Active Set method with improved normal compliance produces more fluctuations which can be explained by energy dissipation during the impact, but the energy of the system is recovered and conserved after the impact.
\begin{figure}[!h]
\begin{center}
\includegraphics[width=16cm]{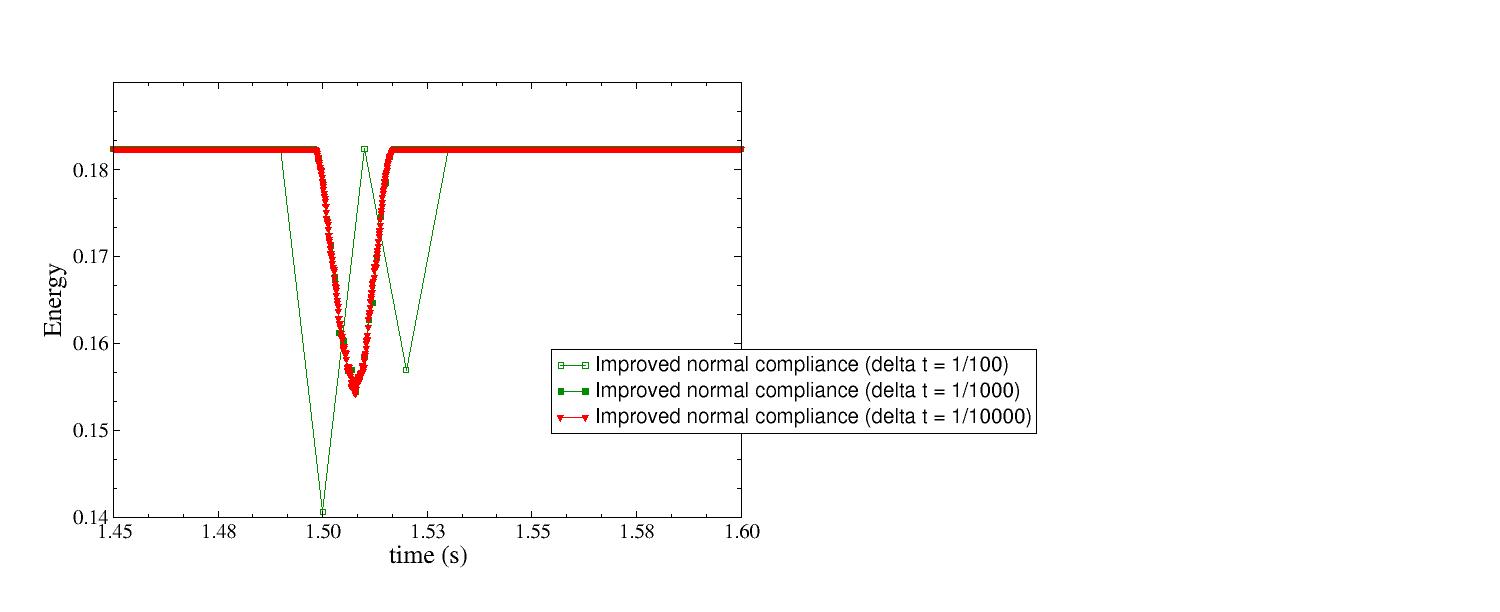}
\end{center}
\caption{Discrete energy behavior of the Active Set scheme with Improved Normal compliance ($\alpha =3$) as a function of $\Delta t$ ($c_\nu={1 e^{3}}$).}
\label{ene_ball_3}
\end{figure}
In Figure \ref{ene_ball_3}, we analyze the discrete energy behavior of the Active Set scheme for the improved normal compliance condition for different time steps. For $\Delta t$ varying from $1e^{-2}\ s$ to $1e^{-4}\ s$, we can notice that the obtained numerical results are similar and this method is characterized by a dissipation of
energy during the impact, but after the impact the energy of the system is conserved. For $\Delta t=1e^{-2}\ s$, the method generates even more fluctuations during impact. So taking $\Delta t$ below $1e^{-3}\ s$ does not guarantee the minimization of fluctuations.
\begin{figure}[!h]
\begin{center}
\includegraphics[width=18cm]{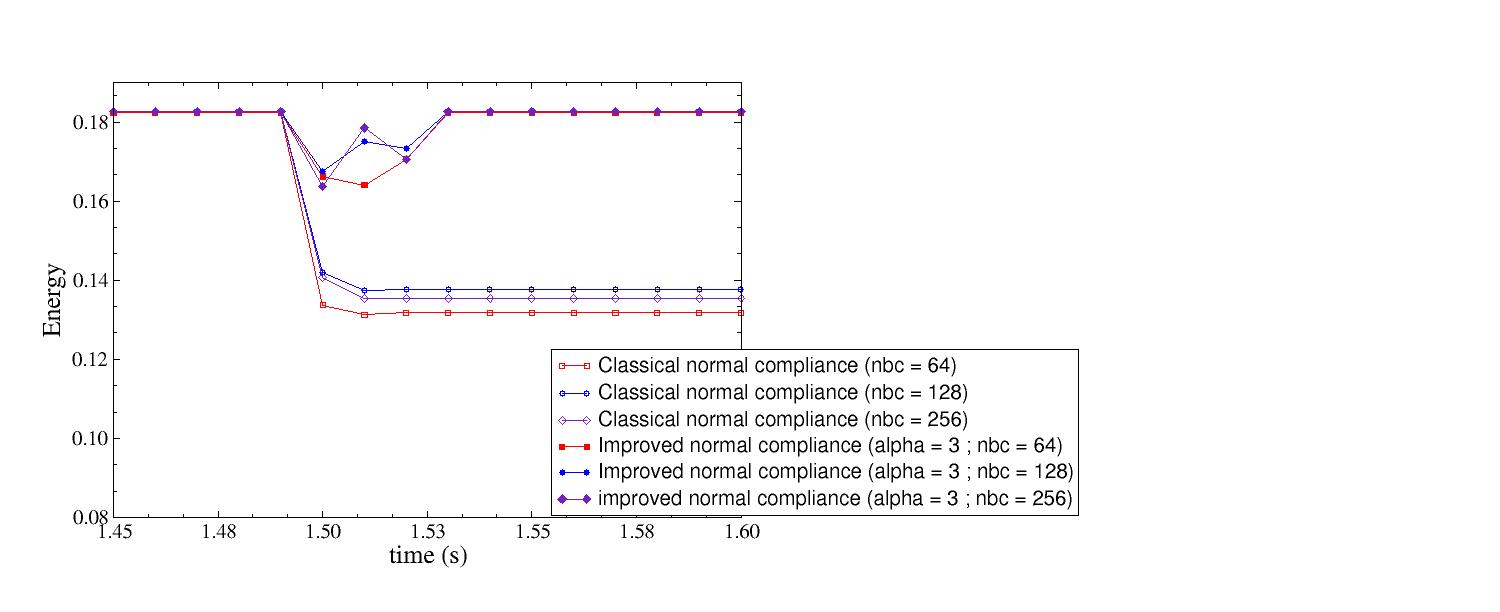}
\end{center}
\caption{Discrete energy behavior of Active Set schemes with classic Normal Compliance and with Improved Normal compliance ($\alpha =3$) as a function of the number of contact points ($\Delta t=0.01$, $c_\nu=1 e^{3}$).}
\label{ene_ball_4}
\end{figure}

In Figure \ref{ene_ball_4}, we observe the behavior of the discrete energy given by the Active Set scheme with classic normal and improved compliance with respects to the parameter $\alpha$ and mesh refinement related to the contact boundary, that is to say the number of contact points $nbc$. We notice that the numerical results obtained by using the classical normal compliance with a different number of contact points show a strong energy dissipation (between $24\%$ and $28\%$). Such a phenomenon is expected in this configuration and can be explained by the discrete energy balance; we refer to \cite{ayyad2009formulation,ayyad2009frictionless} for more details. However, using the Improved Normal Compliance method ($\alpha=3$) with the same contact boundary discretization, we see a remarkable improvement regarding the conservation of  energy in each case after the impact and with less fluctuation during the impact. It is interesting, albeit expected, to observe to what extent the mesh refinement related to the contact boundary impacts the energy conservation properties in both cases.
\begin{figure}[h!]
\begin{center}
	\includegraphics[width=18cm]{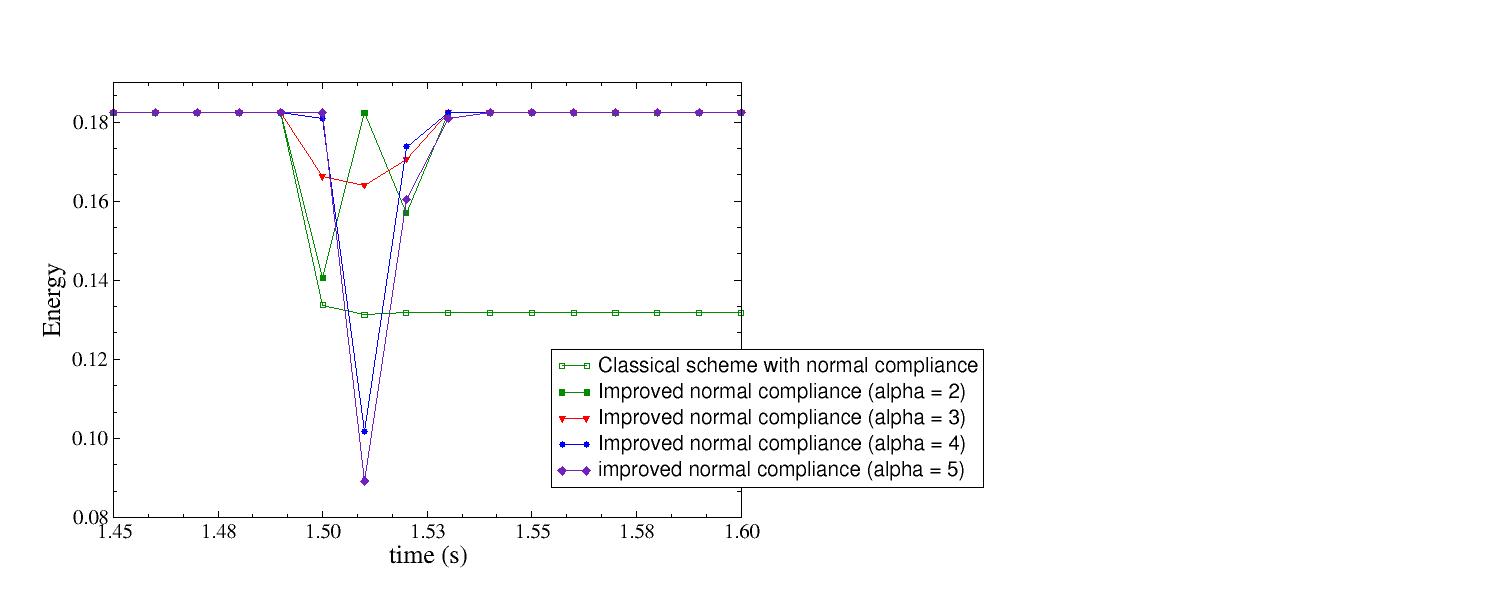}
\end{center}
\caption{Discrete energy behavior of the Active Set scheme with Improved Normal compliance ($\Delta t =0.01\ s$) as a function of $\alpha$ ($c_\nu=1 e^{3}$).}
\label{ene_ball_5}
\end{figure}
\begin{figure}[h!]
\begin{center}
	\includegraphics[width=18cm]{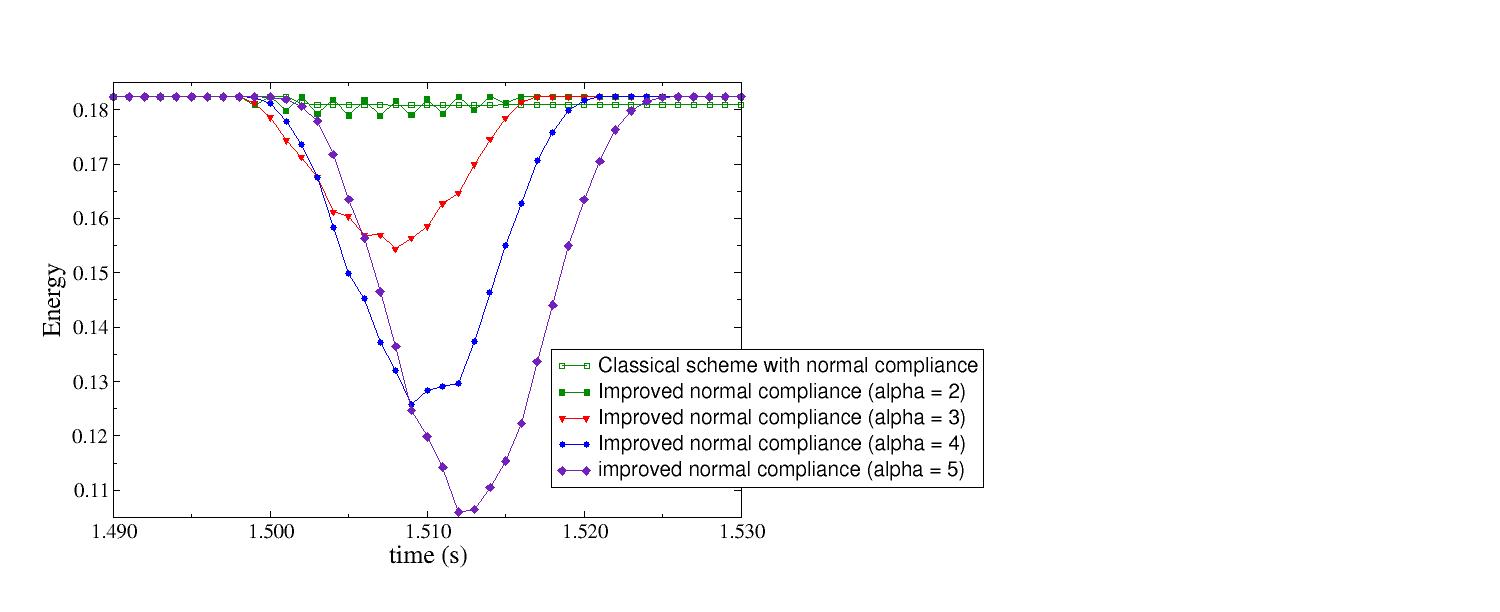}
\end{center}
\caption{Discrete energy behavior of the Active Set scheme with Improved Normal compliance ($\Delta t =0.001$) as a function of $\alpha$ ($c_\nu=1 e^{3}$).}
\label{ene_ball_6}
\end{figure}
\newpage
In Figure \ref{ene_ball_5} with a time step of $1e^{-2}\ s$, we notice that the Active Set scheme produces a significant energy dissipation ($28\%$) with the standard normal compliance method. For the improved normal compliance, we notice some fluctuations, but the system recovers exactly the energy after the impact. By comparing with Figure \ref{ene_ball_6} ($\Delta t=1e^{-3}\ s$) and that of Figure \ref{ene_ball_5}, we notice large fluctuations for different values of $\alpha$ greater than $4$ for the improved normal compliance method. Note that, while it seems at first sight that the classical scheme with normal compliance conserves the energy, that is not actually the case as there is a slight energy dissipation ($1\%$). Such a phenomena is expected and can be explained by the smallness of $\Delta t$ considered.

Regarding the fluctuations, the results observed seem consistent with the theory. As a matter of fact, let us consider the following configurations:
\begin{itemize}
	\item Let $\alpha_1$ and $\alpha_2$ be two values of the $\alpha$ coefficient used in INC, such that $\alpha_2>\alpha_1$;
	\item Let the associated normal stress values $\lambda_{{\boldsymbol \nu-\frac{1}{2}}}^1$, $\lambda_{{\boldsymbol \nu-\frac{1}{2}}}^2$ be given by (\ref{improvement}).
\end{itemize}
As the penetrations are very small for this physical setting, for the same penetration we have $\lambda_{{\boldsymbol \nu-\frac{1}{2}}}^1 > \lambda_{{\boldsymbol \nu-\frac{1}{2}}}^2$. It means that the penetrations observed with $\alpha=\alpha_2$ are to be larger than with $\alpha=\alpha_1$ (see Table \ref{tab}). Therefore, the energy dissipated is expected to be larger with $\alpha=\alpha_2$ and it turns out to be the case with Figure \ref{ene_ball_6}, as the fluctuation amplitude is strictly increasing with respect to the value of $\alpha$.
\begin{table}[!ht]
\begin{center}
	\begin{tabular}{lccccc}
		\hline
		nbc&16&32&64&128&256\\
		\hline
		dof&50&192&784&3160&12520\\
		\hline
		CPU time for Quasi-Lagrangian with Signorini law & 6.2 & 25.6 & 248.9 & 2950.1 & 47452.8 \\
		CPU time for Active set with persistent conditions & 7.2 & 33.7 & 159.2 & 1089.9 & 10492.8 \\
		CPU time for Active set with improved normal compliance & 7.2 & 34.2 & 158.6 & 1101.1 & 10527.1 \\
		CPU time for Active set with classical normal compliance & 7.9 & 33.8 & 159.1 & 1074.5 & 10384.2 \\
		\hline
	\end{tabular}
\end{center}
\caption{Results of the Active Set method and the quasi-Lagrangian method in comparison with the number of degrees of freedom (dof), the number of contact nodes (nbc) and the total CPU time (CPU) in seconds ($\Delta t=0.01\,s$, $\alpha=3$, $c_\nu=1 e^{3}$).} \label{tabb}
\end{table}

In Table \ref{tabb}, we study the convergence in CPU time of the Active Set method compared to different methods (persistent contact, classic normal compliance and the quasi-Lagrangian approach for unilateral contact), depending on the number of degrees of freedom (dof) and the number of contact points (nbc) on the boundary $\Gamma_3$.
When $\text{nbc}<64$, the quasi-Lagrangian method is slightly more efficient than the Active Set method. When we increase the number of degrees of freedom (dof) we notice that the Active Set method is much faster in terms of CPU time than the quasi-Lagrangian method; this can be explained by the fact that the Active Set method does not require the use of Lagrange multipliers, and this method requires less Newton iterations; for more details we refer for instance to \cite{ABD1,ABCD2}.
\subsection{Impact of a hyperelastic ring on a foundation}
The interest of this example is to give a validation of the Active Set method with the improved normal compliance conditions on a hyperelastodynamic problem. This non trivial example, introduced by Laursen \cite{laur2002} concerns an academic problem of frictional impact of a hyperelastic ring against a foundation.
\begin{figure}[!h]
	\begin{center}
		\includegraphics[width=10cm]{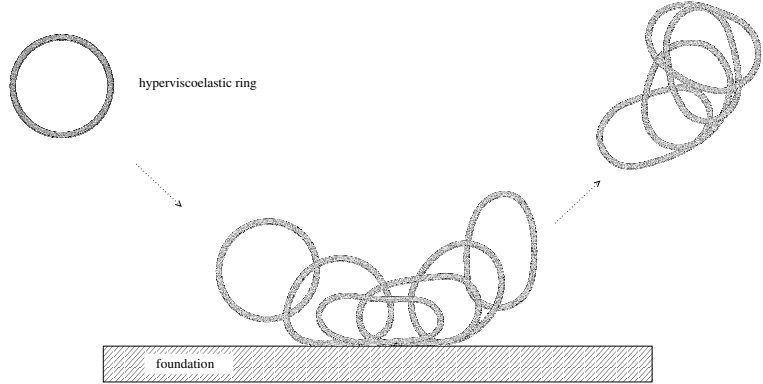}
	\end{center}
	\caption{Sequence of the deformed hyperelastic ring before, during, and after impact.}
	\label{ring_test}
\end{figure}
The details on the physical setting of the problem are given below:
\begin{align*}
	&\Omega\ , =\left\{(x_1,x_2)\in \mathbb{R}^2:\ 81\leq(x_1-100)^2
	+ (x_2-100)^2 \leq 100 \right\},\\
	& \partial_0 \Omega= \varnothing, \qquad \partial_g \Omega= \left\{(x_1,x_2)\in \mathbb{R}^2: \ (x_1-100)^2
	+ (x_2-100)^2 = 81 \right\}, \\
	& \partial_c \Omega=\left\{(x_1,x_2)\in \mathbb{R}^2: \ (x_1-100)^2
	+ (x_2-100)^2 = 100 \right\}.
\end{align*}
The domain $\Omega$ represents the cross-section of a three-dimensional deformable body under the assumption of plane stress. The ring is launched with an initial velocity  towards a foundation, as shown in Figure \ref{ring_test}. The foundation is given by $\left\{(x_1,x_2)\in \mathbb{R}^2:\ x_2 \leq 0 \right\}$. For the discretization, we use 1664 elastic nodes. For numerical experiments, the data are:
\begin{eqnarray}
	\begin{array}{l}
		\rho = 1000 \,kg/m^3, \quad T=10\,s, \quad \Delta t=0.01\,s, \nonumber \\[2mm]
		{\bf u}_0=(0,0)\,m, \quad {\bf u}_1=(10,-10)\,m/s, \quad {\bf f}_0=(0,0)\,N/m ^2, \quad {\bf f}_1= (0,0)\, N/m, \nonumber \\[2mm]
		c_1 = 0.5\,M\!Pa, \quad c_2 = 5.0 e^{-3}\,M\!Pa, \quad c_3 = 5.0 e^{-5}\,M\!Pa, \quad D = 100\,M\!Pa,  \nonumber \\[2mm]
		c_{\nu}=1000, \quad a=1.5 ,\quad b=0.5, \quad \alpha =100. \nonumber
	\end{array}
\end{eqnarray}
The response of the compressible material, considered for Ogden's constitutive law \cite{ciarlet1982lois} is characterized by the following energy density:
\begin{equation}\nonumber
	W({\bf F}) = c_1 (I_1 - 3) + c_2 (I_2 - 3) + d (I_3 - 1) - (c_1 + 2 c_2 + d) \ln
	I_3,
\end{equation}
with $c_1 = 0.5 \,M\!Pa$, $c_2 = 0.5 \cdot 10^{-2}\, M\!Pa$ and $d = 0.35 \,M\!Pa$ and, using the Green-Lagrange tensor defined by $\mathbf C = {\mathbf F}^T \mathbf F$, the invariants $I_1$, $I_2$ and $I_3$ are defined by
$$I_1({\bf C}) ={\rm tr}({\bf C}), \qquad I_2({\bf C}) = \frac{(\rm tr({\bf C}))^2 - tr({\bf C}^2)}{2}, \qquad I_3({\bf C}) = \det({\bf C}). $$
\begin{figure}[!h]
	\begin{center}
		\includegraphics[width=11cm]{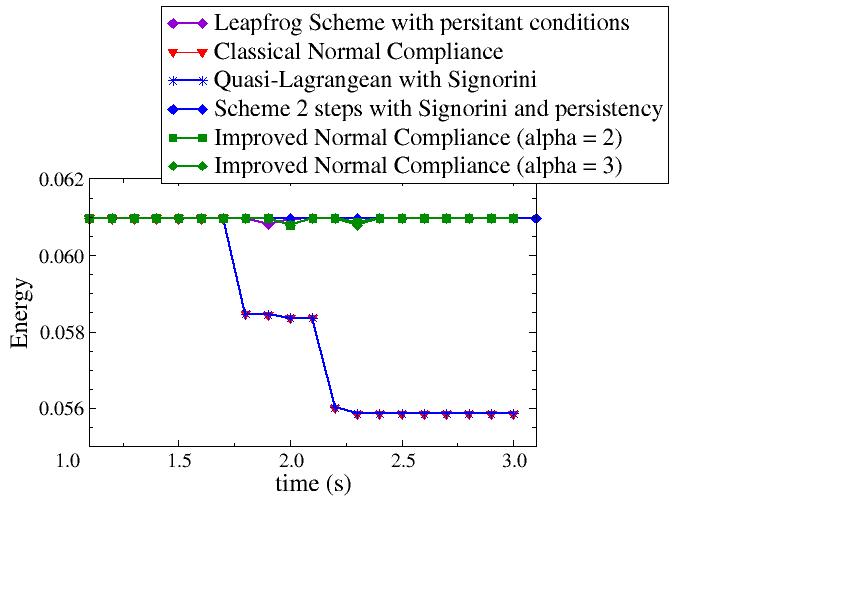}
	\end{center}
	\caption{Discrete energy behavior of selected time integration schemes during frictionless impact ($\Delta t=0.01$, $c_\nu=1 e^{3}$).}
	\label{ene_ring_1}
\end{figure}
\begin{table}[!ht]
	\begin{center}
		\begin{tabular}{lc}
			\hline
			Methods & Maximum error on the $\delta_\nu$ \\
			\hline
			Quasi-Lagrangian with Signorini law & 0. \\
			Scheme with Signorini and persistent (Ayyad--Barboteu) & $7.1 e^{-2}$ \\
			Active set with persistent conditions &$8.9 e^{-2}$\\
			Active set with classical normal compliance & $1 e^{-5}$ \\
			Active set with improved normal compliance ($\alpha = 2$) &  $3.4 e^{-4}$ \\
			Active set with improved normal compliance ($\alpha = 3$)& $4.9 e^{-3}$ \\
			\hline
		\end{tabular}
	\end{center}
	\caption{Maximum error on normal contact displacement ($\Delta t=1 e^{-2}$, $c_\nu=1 e^{3}$).} \label{tabbb}
\end{table}

\begin{table}[!ht]
	\begin{center}
	\scalebox{0.9}{
		\begin{tabular}{lccccc}
			\hline
			nbc&32&64&128&256&512\\
			\hline
			dof&192&384&1792&4608&15360\\
			\hline
			CPU time for Quasi-Lagrangian with Signorini law & 8.59 & 18.03 & 139.69 & 428.80 & 2693.86 \\
			CPU time for Active set with persistent conditions & 7.3 & 16.2 & 99.5 & 301.8 & 1252.7 \\
			CPU time for Active set with improved normal compliance $\alpha=3$ & 6.9 & 17.3 & 107.0 & 274.7 & 1251.8 \\
			CPU time for Active set with classical normal compliance & 6.3 & 13.4 & 87.9 & 221.9 & 1056.8 \\
			\hline
		\end{tabular}
		}
	\end{center}
	\caption{Results of the Active Set method and the quasi-Lagrangian method in comparison with the number of degrees of freedom (dof), the number of contact nodes (nbc) and the total CPU time (CPU) in seconds ($\Delta t=0.01$,~$c_\nu=1 e^{3}$).} \label{tabbbb}
\end{table}

\begin{figure}[!h]
	\begin{center}
		\includegraphics[width=10cm]{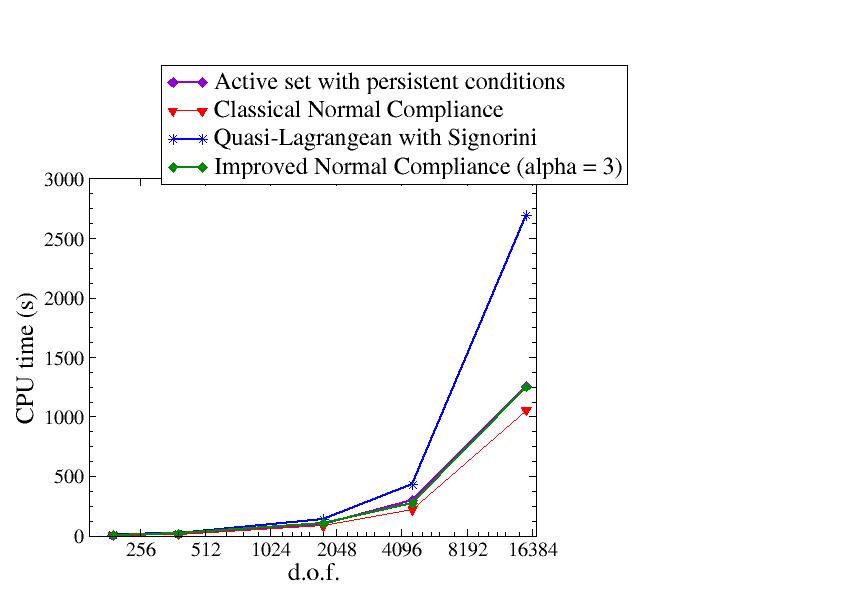}
	\end{center}
	\caption{CPU time of the Active Set method and the quasi-Lagrangian method in comparison with the number of degrees of freedom ($\Delta t=0.01$, $c_\nu=1 e^{3}$).}
	\label{CPU_ring}
\end{figure} 

\begin{figure}[!h]
	\begin{center}
		\includegraphics[width=18cm]{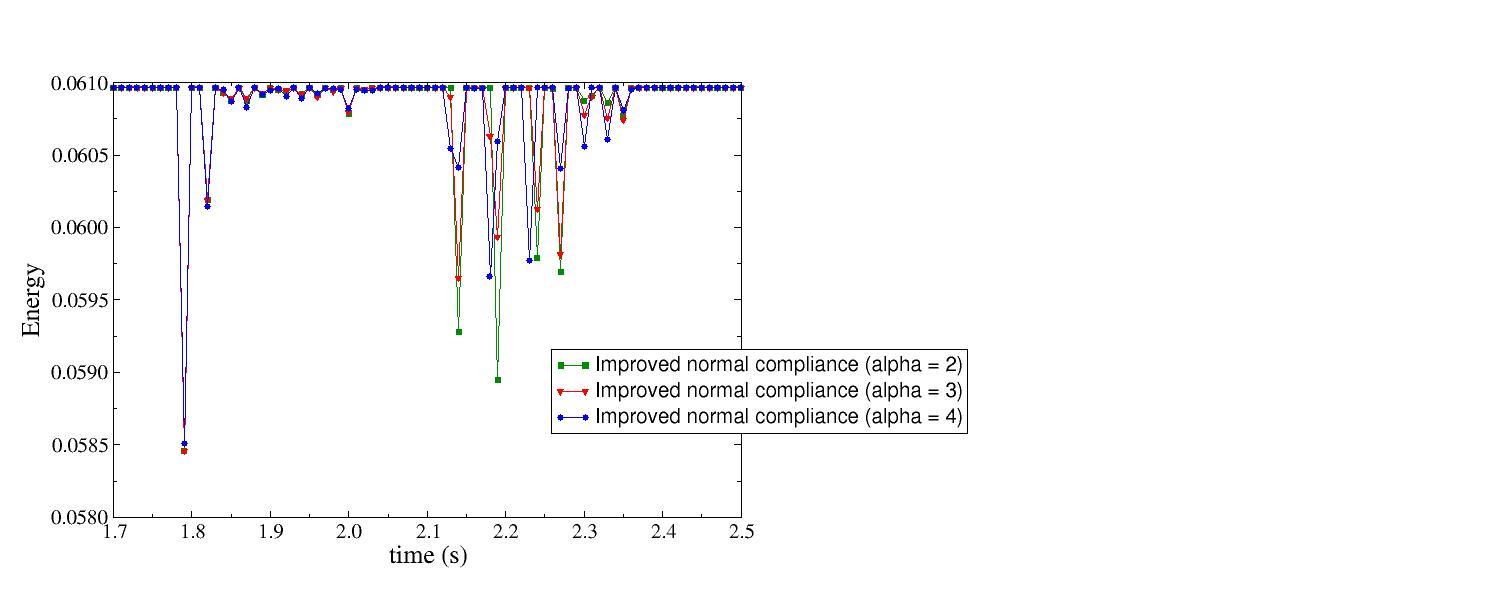}
	\end{center}
	\caption{Discrete energy behavior of the Active Set scheme with Improved Normal compliance ($\Delta t=0.01$) as a function of $\alpha$ without friction ($c_\nu=1 e^{3}$).}
	\label{ene_ring_2}
\end{figure}

In Figure \ref{ene_ring_1}, we observe the evolution of the total discrete energy of the frictionless dynamic system depending on selected time-integration schemes, i.e, leapfrog scheme with persistent conditions, classical normal compliance, quasi-Lagrangian method, two-step scheme \cite{ayyad2009formulation} and  improved normal compliance. We find that the leapfrog scheme and the two-step scheme (Signorini and persistent condition) conserve the energy of the system unlike the quasi-Lagrangian method with unilateral contact ($9\%$ of energy dissipation) and the classical normal compliance ($9\%$ of energy dissipation) which generates maximum error on the normal contact displacement of $1e^{-5} m$ ($c_\nu = 1e^4$). Due to the increment step of the leapfrog scheme, this method generates a maximum error on the normal contact displacement of $8.9e^{-2}m$ (see Table \ref{tabbb}), therefore higher than the other methods. Moreover, the improved normal compliance method is characterized by an exact conservation of energy after the impact. This method also allows to minimize interpenetration with $\alpha=2$ ($3.4e^{-4} m$, see Table \ref{tabbb}).

In Figure \ref{ene_ring_2}, with a time step of $1e^{-2}\,s$, we present the evolution of the total discrete energy of the frictionless dynamic system with the Active Set scheme with improved normal compliance conditions according to the parameter $\alpha$. We observe an exact conservation of energy after the impact. However, the method generates some fluctuations during impact. Based on these observations, the choice $\alpha=3$ seems to perform at best.

In Table \ref{tabbbb}, we study the convergence in CPU time of the Active Set--INC method compared with different methods (persistent contact, standard normal compliance  and quasi-Lagrangian method for unilateral contact) depending on the number of degrees of freedom (dof) and  the number of contact points (nbc) on the boundary $\Gamma_3$. For different values of nbc, the Active Set--INC method is better than the quasi-Lagrangian method in terms of CPU time; as in the elastic ball numerical example, this can be explained by the fact that the Active Set--INC method does not involve Lagrange multipliers and that it generates fewer nonlinear Newton iterations.\\

\begin{figure}[!h]
	\begin{center}
		\includegraphics[width=18cm]{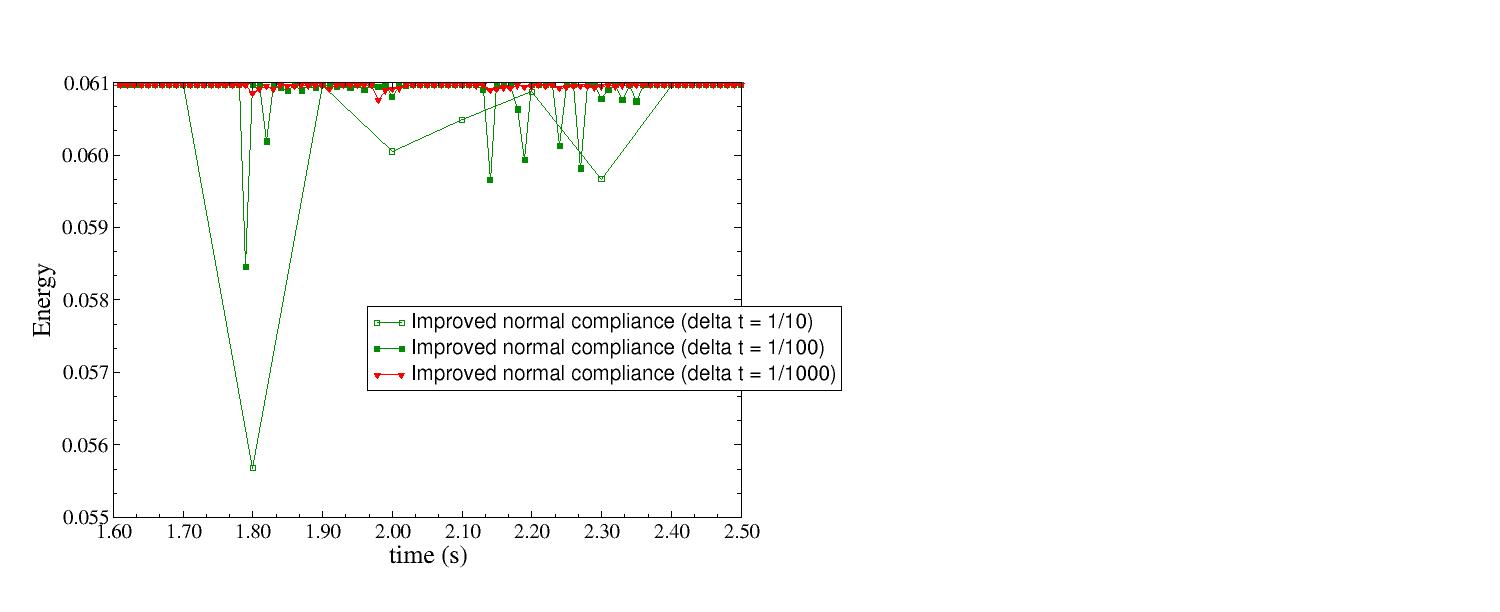}
	\end{center}
	\caption{Discrete energy behavior of the Active Set scheme with Improved Normal compliance ($\alpha =3$) as a function of $\Delta t$ ($c_\nu=1 e^{3}$).}
	\label{ene_ring_4}
\end{figure}
In Figure \ref{ene_ring_4}, we are interested in the Active Set--INC method ($\alpha=3$) depending on the time step. For $\Delta t=1e^{-3}\ s$, we notice that the method is characterized by a better conservation of energy during the impact. On the other hand, for a time step $\Delta t$ ranging from $1e^{-1}\, s$ to $1e^{-2}\,s$, we notice that the method generates some energy fluctuations during the impact, but it conserves energy after the impact.

In Figure \ref{ene_ring_3}, we observe the behavior of the discrete energy using an Active Set scheme with normal compliance condition with friction ($\mu=0.2$). This method is characterized by a slightly dissipative behavior of friction; this is consistent with the physical dissipative nature of the friction phenomenon. The difference of the dissipative behavior comes from the model of frictional contact according to the parameter $\alpha$.
\begin{figure}[!h]
	\begin{center}
		\includegraphics[width=18cm]{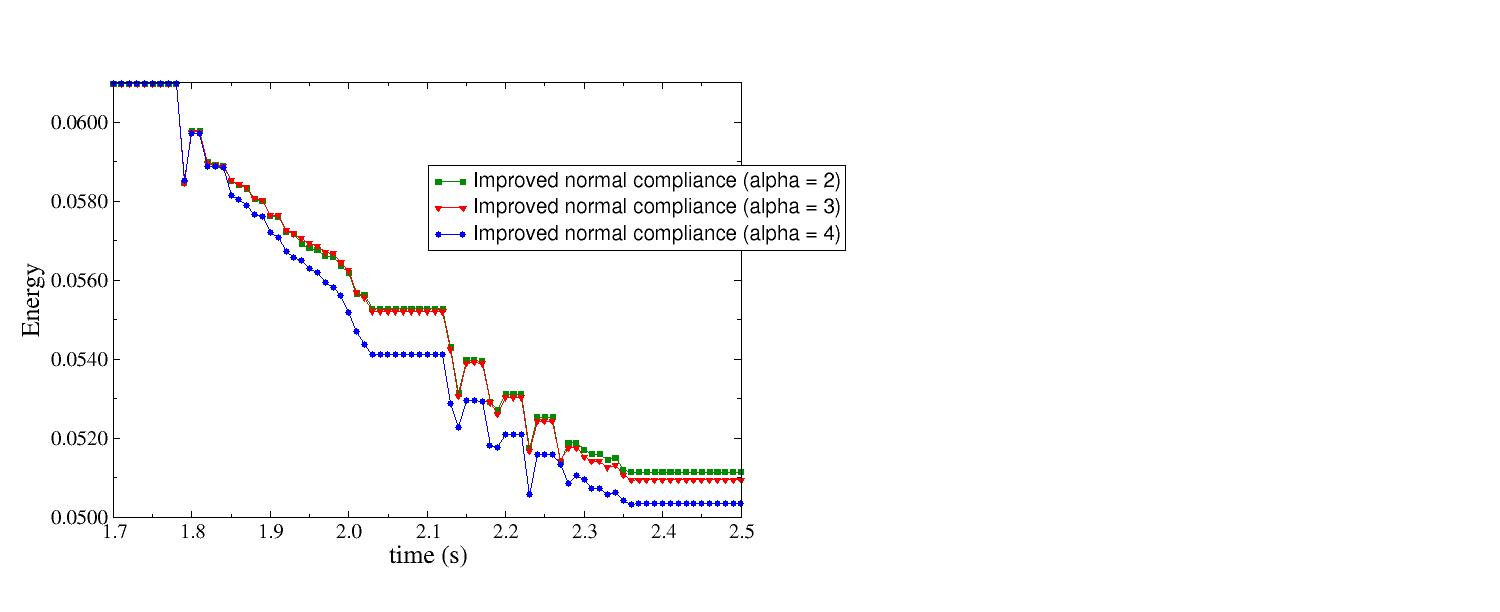}
	\end{center}
	\caption{Discrete energy behavior of the Active Set scheme with Improved Normal compliance ($\Delta t =1e^{-2}\,s$) as a function of $\alpha$ with friction ($\mu = 0.2$, $c_\nu=1 e^{3}$, $c_\tau=1 e^{3}$).}
	\label{ene_ring_3}
\end{figure}

\section{Conclusion and perspectives} \label{s7}
In this work, we investigated a new energy conservation method for hyperelastodynamic contact problems, from a theoretical and numerical point of view.  We first introduced the mathematical framework for general frictional contact problems in large deformations, and demonstrated how the Improved Normal Compliance (INC) condition can be applied in order to obtain energy conservation properties during impact. Next, we assessed these conservation properties in the continuous case and in its space-time discrete counterpart. Then, we derived two semi-smooth Newton Primal Dual algorithms to handle respectively the Standard Normal Compliance (SNC) and INC conditions. The following section presents numerical results on two classical representative academic test cases, namely the impact of an elastic ball and the impact of a hyperelastic ring on a foundation. The aim was to compare the performances of the INC with respect to the SNC and the classical methods used within the literature. It turned out that the methods implemented are at least just as much relevant as several other methods, regarding energy conservation properties after impact, as they display physically realistic behaviors. In the second test case, even more challenging from a numerical point of view, the energy is even almost perfectly conserved during and after the impact. Also, one of the interesting and well-known advantages when using the Active Set approach is the numerical efficiency. Indeed, in both test cases, this method is much more efficient (CPU time-wise, between twice and five times as fast for large problems) than the classical quasi-Lagrangian approach.

At this stage, we outline some future perspectives. As the results in 2D cases seem quite promising, an extension to 3D ought to be considered in order to assess the behavior of this approach on more complex cases. We can reasonably expect the performances with respect to the quasi-Lagrangian approach in 3D to be of an even greater order of magnitude than in 2D. Regarding an underlying practical use case, it could be interesting to investigate a potential application to granular media. Finally, this method was only applied for hyperelastic and elastic material so far, and provided quite successful results, which suggests to consider alternative constitutive laws. Amongst them, in particular, an extension to viscosity and plasticity would make sense as such materials involves energy dissipation phenomena.

\bibliographystyle{unsrt}
\bibliography{sample}

\end{document}